\documentclass[12pt]{article}

\usepackage{amsmath,amsfonts,amsthm,amscd, amssymb,amsrefs}
\usepackage{latexsym}
\usepackage{pictexwd,dcpic}
\usepackage[all,cmtip]{xy}
\usepackage{xcolor}
\usepackage{graphicx}

\allowdisplaybreaks
\newtheorem{thm}{Theorem}
\newtheorem{prop}[thm]{Proposition}
\newtheorem{cor}[thm]{Corollary}
\newtheorem{lemma}[thm]{Lemma}

\newtheorem{example}[thm]{Example}
\newtheorem{rem}[thm]{Remark}
\newtheorem{defn}[thm]{Definition}

\newcommand{\ZZ}{{\mathbb{Z}}}
\newcommand{\QQ}{{\mathbb{Q}}}

\newcommand{\NN}{{\mathbb{N}}}

\newcommand{\ghu}{{g_{h,u}}}
\newcommand{\gtu}{{g_{t,u}}}
\newcommand{\gkut}{{g_{k,ut}}}
\newcommand{\hkv}{{h_{k,v}}}
\newcommand{\ghkvu}{{g_{h_{k,v},u}}}
\newcommand{\sshu}{{\sigma_{h}(u)}}
\newcommand{\sskut}{{\sigma_{k}(ut)}}
\newcommand{\sstu}{{\sigma_{t}(u)}}
\newcommand{\sskv}{{\sigma_{k}(v)}}
\newcommand{\sshkvu}{{\sigma_{h_{k,v}}(u)}}
\newcommand{\hp}{{h^\prime}}
\newcommand{\xp}{{x^\prime}}
\newcommand{\yp}{{y^\prime}}

\newcommand{\gphu}{{g^\prime_{h,u}}}
\newcommand{\ghpu}{{g_{h^\prime,u}}}
\newcommand{\ghphu}{{g_{h^\prime h,u}}}
\newcommand{\ghsshpu}{{g_{h, \sigma_{h^\prime}(u)}}}
\newcommand{\fp}{{f^\prime}}
\newcommand{\fpp}{{f^{\prime\prime}}}
\newcommand{\sshpu}{{\sigma_{h^\prime}(u)}}
\newcommand{\sshphu}{{\sigma_{h^\prime h}(u)}}
\newcommand{\sshsshpu}{{\sigma_h\big(\sigma_{h^\prime}(u)\big)}}

\newcommand{\up}{{u^\prime}}

\newcommand{\sh}{{_{_H}}}

\newcommand{\Ind}{\,\mathrm{Ind}}
\newcommand{\Res}{\,\mathrm{Res}}

\newcommand{\ssres}{\,\mathrm{res}}

\newcommand{\id}{\,\mathrm{id}}
\newcommand{\op}{\,\mathrm{op}}
\newcommand{\Sd}{\,\mathrm{Sd}}
\newcommand{\ch}{\,\mathrm{ch}}
\newcommand{\Hom}{\,\mathrm{Hom}}

\def\dom{\backslash}
\def\rouge{}
\def\magenta{}
\definecolor{darkgreen}{rgb}{0,0.8,0.5}
\def\dgreen{}

\begin{document}

\title{Monomial $G$-posets and their Lefschetz invariants}
\author{Serge Bouc and Hatice Mutlu}
\date{}
\maketitle
\begin{abstract}
Let $G$ be a finite group, and $C$ be an abelian group. We introduce the notions of $C$-monomial $G$-sets and $C$-monomial $G$-posets, and state some of their categorical properties. This gives in particular a new description of the $C$-monomial Burnside ring $B_C(G)$. We also introduce Lefschetz invariants of $C$-monomial $G$-posets, which are elements of $B_C(G)$. These invariants allow for a definition of a generalized tensor induction multiplicative map $\mathcal{T}_{U,\lambda}: B_C(G)\to B_C(H)$ associated to any $C$-monomial $(G,H)$-biset $(U,\lambda)$, which in turn gives a group homomorphism $B_C(G)^\times\to B_C(H)^\times$ between the unit groups of $C$-monomial Burnside rings.\medskip\par\noindent
{\small\bf AMS Classification: } 06A11, 19A22, 20J15\\
{\small \bf Keywords:} Burnside ring, monomial, tensor induction, Lefschetz invariant
\end{abstract}
\section{Introduction}
Let $G$ be a finite group, and $C$ be an abelian group. In this work, we first introduce the notion of {\em $C$-monomial $G$-set}: this is a pair $(X,\mathfrak{l})$ consisting of a finite $G$-set $X$, together with a functor from the transporter category $\widehat{X}$ of $X$, to the the groupoid $\bullet_C$ with one object and automorphism group $C$. The $C$-monomial $G$-sets form a category $_CMG\hbox{-\sf set}$, and we show that it is equivalent to the category $_CFG\hbox{-\sf set}$ of $C$-fibred $G$-sets considered by Barker (\cite{FIB}). In particular, the $C$-monomial Burnside ring $B_C(G)$ introduced by Dress (\cite{MON}) is isomorphic to the Grothendieck ring of the category $_CMG\hbox{-\sf set}$.\par
We extend these definitions to the notion of {\em $C$-monomial $G$-poset}: this is a pair $(X,\mathfrak{l})$ consisting of a finite $G$-poset $X$, and a functor $\mathfrak{l}$ from the transporter category $\widehat{X}$ to $\bullet_C$. We associate to each such pair $(X,\mathfrak{l})$ a Lefschetz invariant $\Lambda_{(X,\mathfrak{l})}$ lying in $B_C(G)$. We show that any element of $B_C(G)$ is equal to the Lefschetz invariant of some (non unique) $C$-monomial $G$-poset.\par
We also introduce the category $_CMG\hbox{-\sf poset}$ of $C$-monomial $G$-posets, and show that there are natural functors of induction $\Ind_H^G: {_CMH\hbox{-\sf poset}}\to{_CMG\hbox{-\sf poset}} $ and of restriction $\Res_H^G: {_CMG\hbox{-\sf poset}}\to {_CMH\hbox{-\sf poset}}$, whenever $H$ is a subgroup of $G$. These functors are compatible with the construction of Lefschetz invariants.\par
We extend several classical properties of the Lefschetz invariants of $G$-posets to Lefschetz invariants of $C$-monomial $G$-posets (the classical case being the case where $C$ is trivial).\par
We next turn to the construction of generalized {\em tensor induction functors} $$T_{U,\lambda}:{_CMG\hbox{-\sf poset}}\to {_CMH\hbox{-\sf poset}}$$
associated, for arbitrary finite groups $G$ and $H$, to any $C$-monomial $(G,H)$-biset $(U,\lambda)$. We show that these functors induce well defined tensor induction maps
$$\mathcal{T}_{U,\lambda}:B_C(G)\to B_C(H),$$
which are not additive in general, but multiplicative and preserve identity elements. In particular, we get induced group homomorphisms between the corresponding unit groups of monomial Burnside rings, similar to those obtained by Carman (\cite{IND}) for other usual representation rings.\par
We show moreover that under an additional assumption, these tensor induction functors and their associated tensor induction maps are well behaved for composition. This yields to a (partial) fibred biset functor structure on the group of units of the monomial Burnside ring.
\section{The monomial Burnside ring}
Let $G$ be a finite group  and $C$ be an abelian group which is noted multiplicatively. We denote by $G\hbox{-\sf set}$ the category of finite $G$-sets (with $G$-equivariant maps as morphisms), and $B(G)$ the usual Burnside ring of $G$, i.e. the Grothendieck ring of $G\hbox{-\sf set}$ for relations given by disjoint union decompositions of finite $G$-sets.
\subsection{The category of $C$-fibred $G$-sets}
A {\em $C$-fibred $G$-set}
is defined to be a $C$-free $(C\times G)$-set with finitely many $C$-orbits. Let $_CFG$-{\sf set} denote the category 
of $C$-fibred $G$-sets where morphisms are $(C\times G)$-equivariant maps.
The coproduct of $C$-fibred $G$-sets $X$, $Y$ is their coproduct $X\sqcup Y$ as  sets, with the obvious $(C\times G)$-action.
If $X$ and $Y$ are $C$-fibred $G$-sets, there is a $C$-action on $X\times Y$ defined
 by $c(x,y)= (cx,c^{-1}y)$ for any $c\in C$ and $(x,y) \in X\times Y$. 
 The $C$-orbit of an element $(x,y)$ of $X\times Y$ is denoted by $x\otimes y$ 
 and the set of $C$-orbits is denoted by $X\otimes Y$. Moreover $C\times G$ acts on $X\otimes Y$ by 
$$(c,g)(x\otimes y)=cgx\otimes gy$$
for any $(c,g) \in C\times G$ and $x\otimes y\in X\otimes Y$. 
One checks easily that $X\otimes Y$ is again a $C$-fibred $G$-set, 
called the tensor product of $X$ and $Y$.

We denote the isomorphism class of a $C$-fibred $G$-set $X$ by $[X]$.
The {\em $C$-monomial Burnside ring} $B_{C}(G)$, introduced by Dress (\cite{MON}), is defined as the Grothendieck group of the category of $C$-fibred $G$-sets, for relations given by $[X]+[Y]= [X\sqcup Y]$.  The ring structure of $B_C(G)$ is induced by $[X]\cdot[Y]= [X\otimes Y]$.  
The identity element is the set $C$ with trivial $G$-action and the zero element is
the empty set. If $C$ is trivial we recover the ordinary Burnside ring of the group $G$.\par
Given a $C$-fibred $G$-set $X$, we denote the set of $C$-orbits on $X$ by  $C\dom X$.
 The group~$G$ acts on $C\dom X$, and $X$ is $(C\times G)$-transitive if and only if $C\dom X$ is $G$-transitive.
 If $C\dom X$ is transitive as a $G$-set it is isomorphic to $G/U$ for some $U\leq G.$ 
 There exists a group homomorphism $\mu :U \rightarrow C$ such that if $U$ is the stabilizer
  of the orbit $Cx$, then $ax=\mu(a)x$ for all $a\in U$. Since  the stabilizer $(C\times G)_x$ of $x$ in $C\times G$ is equal to 
$$(C\times G)_x=\{\big(\mu(a)^{-1},a\big)\mid a\in U \},$$
the {\rouge $C$-fibred} $G$-set $X$ is determined up to isomorphism 
by the subgroup $U$ and $\mu$. \par
Conversely, let $U$ be a subgroup of $G$, and 
$\mu: U \rightarrow C$ be a group homomorphism. Then we set $U_\mu=\{\big(\mu(a)^{-1},a\big)\mid a\in U\}$, and denote by $[U,\mu]_G$ the {\rouge $C$-fibred} $G$-set $(C \times G)/U_\mu$. The pair $(U,\mu)$ is 
called a $C$-subcharacter of $G$. We denote the set of $C$-subcharacters
by $\ch(G)$. The group $G$ acts on $\ch(G)$ by conjugation. The $G$-set $\ch(G)$
is a poset with the relation $\leq$ defined by 
$$(U,\mu) \leq (V,\nu) \Leftrightarrow U \leq V \, \text{and} \, \Res^V_U\nu=\mu$$
for any $(U,\mu)$ and $(V,\nu)$ in $\ch(G)$.\par

As an abelian group we have 
$$B_{C}(G)= \bigoplus\limits _{(U,\mu) \in_{G} \ch(G) }\ZZ [U,\mu]_G$$
where $(V , \nu)$ runs over $G$-representatives 
of the $C$-subcharacters of $G$, details can be seen in \cite{FIB}.

\subsection{The category of $C$-monomial $G$-sets}
Let $G$ be a finite group and $C$ be an abelian group.
Given a $G$-set $X$, we consider its transporter category $\widehat{X}$
whose objects are the elements of $X$ and given $x$, $y$ in $X$
 the set of morphisms from $x$ to $y$ is 
$$\Hom_{\widehat{X}}(x, y)= \{ g \in G \mid gx=y\}.$$
Let $\bullet_{C}$  denote the category with one object
 where morphisms are the elements of $C$ and composition is multiplication in $C$.
Now we define $C$-monomial $G$-sets as follows.
\begin{defn}
A {\em $C$-monomial $G$-set} is a pair $(X, \mathfrak{l})$ consisting of a finite $G$-set $X$ and a functor $\mathfrak{l}: \widehat{X} \rightarrow \bullet_{C}$.
\end{defn}
{\dgreen In otherwords, for each $x,y\in X$ and $g\in G$ such that $gx=y$, we have an element $\mathfrak{l}(g,x,y)$ of $C$, with the property that $\mathfrak{l}(h,y,z)\mathfrak{l}(g,x,y)=\mathfrak{l}(hg,x,z)$ if $h\in G$ and $hy=z$, and $\mathfrak{l}(1,x,x)=1$ for any $x\in X$.}\par
Let $(X, \mathfrak{l})$ and $(Y, \mathfrak{m})$ be $C$-monomial $G$-sets. If $f:X\to Y$ is a map of $G$-sets, we slightly abuse notation and also denote by $f:\widehat{X}\to\widehat{Y}$ the obvious functor induced by~$f$. Now a {\em map $(f, \lambda) : (X, \mathfrak{l}) \rightarrow (Y, \mathfrak{m}) $ 
of $C$-monomial $G$-sets} is a pair consisting of a map $f:X\to Y$ of $G$-sets
 and  a natural transformation $\lambda : \mathfrak{l} \rightarrow \mathfrak{m} \circ f$. 
We denote by $_CMG$-{\sf set} the category whose objects are $C$-monomial $G$-sets, 
morphisms are the maps of $C$-monomial $G$-sets, and composition is the obvious one.\par
Let $(X, \mathfrak{l})$ and $(X^\prime, \mathfrak{l}^\prime)$ be $C$-monomial 
$G$-sets. We define the disjoint union of $C$-monomial $G$-sets as 
$(X, \mathfrak{l}) \sqcup (X^\prime, \mathfrak{l}^\prime)
= (X\sqcup X^\prime, \mathfrak{l}\sqcup \mathfrak{l}^\prime)$
where $X\sqcup X^\prime$ is the disjoint union of $G$-sets and 
$$\mathfrak{l}\sqcup \mathfrak{l}^\prime: 
\widehat{X\sqcup X^\prime} \rightarrow \bullet_{C}$$
is the functor such that
\[
(\mathfrak{l}\sqcup \mathfrak{l}^\prime)(g,z_1,z_2) = 
\begin{cases}
\mathfrak{l}(g,z_1,z_2) & z_1,\,z_2 \in X \\
\mathfrak{l}^\prime(g,z_1,z_2) &   z_1,z_2 \in  X^\prime
\end{cases}
\]
for any $z_1,\,z_2 \in X\sqcup X^\prime$ such that $gz_1=z_2$ for
some $g\in G$.\par
The product of $C$-monomial $G$-sets 
$(X, \mathfrak{l}),\,(X^\prime, \mathfrak{l}^\prime)$  is defined 
to be $(X \times X^\prime, \mathfrak{l} \times \mathfrak{l}^\prime)$
where $X \times X^\prime$ is the product of $G$-sets and
 $\mathfrak{l} \times \mathfrak{l}^\prime: \widehat{X \times Y}
 	\rightarrow \bullet_{C}$
is the functor defined by

$$(\mathfrak{l} \times \mathfrak{l}^\prime)\big(g,(x, \xp),(y, \yp)\big)
= \mathfrak{l}(g,x,y) \mathfrak{l}^\prime(g,\xp,\yp)$$

for $g \in G$ and $(x, \xp), (y, \yp) \in X \times X^\prime$ 
such that $g(x, \xp) =(y, \yp)$.\par
Our goal is to show that the categories $_CMG$-{\sf set} and $_CFG$-{\sf set} are equivalent.
 For this, we define a functor $F: {_CMG\text{-\sf set}} \rightarrow {_CFG\text{-\sf set}}$ as follows:
 given a $C$-monomial $G$-set $(X,\mathfrak{l})$, we set 
$$F(X,\mathfrak{l})=C\times_\mathfrak{l} X,$$ 
which is the direct product $C\times X$ endowed with the $(C\times G)$-action defined by
 $(k,g)(c,x)= \big(kc\mathfrak{l}(g,x,gx), gx\big)$ for any $(k,g)\in C\times G$ and $(c,x) \in C\times X$. \par
Given a map $(f,\lambda): (X,\mathfrak{l})\rightarrow (Y,\mathfrak{m})$ of $C$-monomial $G$-sets, we define 
$$F(f,\lambda): C\times_\mathfrak{l} X \rightarrow C\times_\mathfrak{m} Y$$ 
by $F(f,\lambda)(c,x)= (c\lambda_x, f(x))$ for any 
$(c,x)\in  C\times_\mathfrak{l} X$. Then $F(f,\lambda)$ is
a $(C\times G)$-map: indeed, given $(k,g)\in C\times G$ and $(c,x)\in C\times X$, we have
\begin{align*}
(k,g)F(f,\lambda)(c,x)&= (k,g)\big(c\lambda_x, f(x)\big)
=  \big(kc\lambda_x\mathfrak{m}(g,f(x),f(gx)),f(gx)\big)\\
&{\rouge = \big(kc\lambda_{gx}\mathfrak{l}(g,x,gx), f(gx)\big)}
= F(f,\lambda)\big(kc\mathfrak{l}(g,x,gx), gx\big)\\
&= F(f,\lambda)\big((k,g)(c,x)\big).\\
\end{align*}
It is clear that $F:{_CMG\text{-\sf set}} \rightarrow {_CFG\text{-\sf set}}$ is a functor. 
\begin{lemma}
	Let $C$ be an abelian group and $G$ be a finite group. Then the above functor
	 $F: {_CMG\hbox{-\sf set}\to{_CFG}\hbox{-\sf set}}$ is an equivalence of categories.
\end{lemma}
\begin{proof}
	We prove that $F$ is fully faithful and essentially surjective. 
	First we show that $F$ is essentially surjective. Given a $C$-fibred $G$-set $X$,
	let $C\backslash X$ be the set of $C$-orbits. Clearly $C\backslash X$ is a $G$-set.
	We define a functor $\mathfrak{l}: \widehat{C\backslash X} \rightarrow \bullet_C$.
	Let $Cx$, $Cy \in C\backslash X$ such that $Cgx= Cy$ for some $g \in G$.
	Then there exists a unique $c \in C$ such that $gx=cy$. We set 
	$\mathfrak{l}(g,Cx,Cy)= c$. We have 
	$F(C\backslash X, \mathfrak{l})= C\times_\mathfrak{l} (C\backslash X)$.
	Now choose a set $[C\backslash X]$ of $G$-representatives of the $G$-action
	on $C\backslash X$. Then for any $x\in X$, there exits a unique 
	$C\sigma_x \in [C\backslash X]$ such that $x \in C\sigma_x.$ Since $X$ is 
	$C$-free, there exists a unique $c_x \in C$ such that $x=c_x\sigma_x.$
	We define a $(C\times G)$-map 
	$f: X \rightarrow C\times_\mathfrak{l} (C\backslash X)$ such that 
	$f(x)= (c_x, C\sigma_x).$ Then
	$$(c,g)f(x)= (c,g)(c_x,C\sigma_x)= 
	{\rouge (c_xc\mathfrak{l}(g,C\sigma_x,Cg\sigma_x), Cg\sigma_x)}
	=(c_xc,Cg\sigma_x)$$
	$$=(c_{cgx},Cg\sigma_x)= f\big((c,g)x\big).$$
	So $f$ is a $(C\times G)$-map and clearly an isomorphism. Thus,
	$F$ is essentially surjective. 
	
	Let $(X,\mathfrak{l})$ and $(Y,\mathfrak{m})$ be
	$C$-monomial $G$-sets. We need to show that the map
	$$\overline{F}: \Hom\big((X,\mathfrak{l}),(Y,\mathfrak{m})\big)
	\rightarrow\Hom\big(F(X,\mathfrak{l}), F(Y,\mathfrak{m})\big)$$ 
	induced by $F$ is surjective and 
	injective. Let $\varphi : C\times_\mathfrak{l}X \rightarrow 
	C\times_\mathfrak{m}Y$ be a $(C\times G)$-map. 
	Given $(1,x) \in C\times_\mathfrak{l}X$, 
	let $\varphi(1,x)=(c_x,z_x)$ 
	for {\rouge $(c_x,z_x)\in C\times Y$.} 
	Since $\varphi$ is a $(C\times G)$-map, we get
	$$\varphi(c, x)= (cc_x, z_x)$$
	 and 
	$$\varphi(1,gx)= \big(c_x\mathfrak{m}(g,z_x,gz_x)\mathfrak{l}^{-1}(g,x,gx), gz_x\big)$$
	for any $c \in C$ and $g\in G$. We define a map
	$$(f,\lambda): (X,\mathfrak{l}) \rightarrow (Y,\mathfrak{m})$$
	 such that
	$f: X \rightarrow Y$ is defined by $f(x)= z_x$ and 
	$\lambda : \mathfrak{l} \rightarrow \mathfrak{m}\circ f $
	is defined by $\lambda_x= c_x$ for any $x \in X$.
	Clearly, $f$ is a $G$-set map.
	Let $x \in X$ and $g\in G$. Then
	$$\mathfrak{m}\big(g,f(x),f(gx)\big)\lambda_x= 
	\mathfrak{m}\big(g,f(x),f(gx)\big)c_x= c_x\mathfrak{m}\big(g,f(x),f(gx)\big)
	\mathfrak{l}^{-1}(g,x,gx)\mathfrak{l}(g,x,gx)$$
	$$= \mathfrak{l}(g,x,gx)c_{gx}= \mathfrak{l}(g,x,gx)\lambda_{gx}.$$
	So $\lambda: \mathfrak{l} \rightarrow \mathfrak{m}\circ f$ is a natural
	transformation and $(f,\lambda)$ is a map of $C$-monomial
	$G$-sets. Thus, {\rouge $\overline{F}(f,\lambda)=\varphi$ and $\overline{F}$
	is surjective.} The injectivity is clear, so $F$ is fully faithful.
	\end{proof}

\begin{prop}
Let $G$ be a finite group. Then $B_C(G)$ is isomorphic to the Grothendieck ring of the category $_CMG$-{\sf set}, for relations given by decomposition into disjoint unions of $C$-monomial $G$-sets and multiplication induced by product of $C$-monomial $G$-sets.
\end{prop}
\begin{proof}
We let $B^1_C(G)$ denote the Grothendieck ring of
the category $_CMG$-{\sf set}.
The equivalence 
$$F: {_CMG\text{-\sf set}} \rightarrow {_CFG\text{-\sf set}}$$
induces a bijection 
$$\widehat{F} : B^1_C(G) \rightarrow B_C(G)$$
such that 
$$\widehat{F}\Big(\big[(X,\mathfrak{l})\big]\Big)= [C\times_\mathfrak{l} X]$$
for any $C$-monomial $G$-set $(X,\mathfrak{l})$. Now we show that 
$\widehat{F}$ is a ring homomorphism. Let $(X_1,\mathfrak{l}_1)$ and
$(X_2,\mathfrak{l}_2)$ be $C$-monomial $G$-sets. Then
$$\widehat{F}\Big(\big[(X_1,\mathfrak{l}_1)\big]+
\big[(X_1,\mathfrak{l}_1)\big]\Big)= 
\widehat{F}\Big(\big[(X_1,\mathfrak{l}_1)\sqcup(X_1,\mathfrak{l}_1)\big]\Big)
= \widehat{F}\Big(\big[(X_1\sqcup X_2, 
\mathfrak{l}_1\sqcup\mathfrak{l}_2)\big]\Big)$$
$$=\big[C\times_{\mathfrak{l}_1\sqcup\mathfrak{l}_2} (X_1\sqcup X_2)\big]
= \big[(X_1,\mathfrak{l}_1)\sqcup(X_2,\mathfrak{l}_2)\big]
=[C\times_{\mathfrak{l}_1}X_1]+[C\times_{\mathfrak{l}_2}X_2].$$

For multiplicativity of $\widehat{F}$ we define a map
$$f: C\times_{\mathfrak{l}_1\times\mathfrak{l}_2} (X_1\times X_2)
\rightarrow (C\times_{\mathfrak{l}_1}X_1)\times_C(C\times_{\mathfrak{l}_2}X_2)$$
such that
$f\big(c,(x_1,x_2)\big)=(c,x_1)\times_C(1,x_2)$.
 Let $(k,g)\in C\times G$ and
$\big(c,(x_1,x_2)\big) \in C\times_{\mathfrak{l}_1\times\mathfrak{l}_2} (X_1\times X_2)$.
 Then
\begin{align*}
(k,g)f\big(c,(x_1,x_2)\big)&= (k,g)\big((c,x_1)\times_C(1,x_2)\big)
= \big((k,g)(c,x_1)\times_C(1,g)(1,x_2)\big)\\
&= \big(kc\mathfrak{l}_1(g,x_1,gx_1),gx_1\big)
\times_C\big(\mathfrak{l}_2(g,x_2,gx_2),gx_2\big)\\
&= (kc\mathfrak{l}_1(g,x_1,gx_1)\mathfrak{l}_2(g,x_2,gx_2),gx_1)
\times_C(1,gx_2)\\
&= f\big(kc\mathfrak{l}_1(g,x_1,gx_1)\mathfrak{l}_2(g,x_2,gx_2),g(x_1,x_2)\big)\\
&= f\big((k,g)(c,(x_1,x_2))\big).\\
\end{align*}
So $f$ is a $(C\times G)$-map and obviously, $f$ is a $(C\times G)$-isomorphism. 
Using $f$ we get 
$$\widehat{F}\big([X_1,\mathfrak{l}_1]\cdot[X_2,\mathfrak{l}_2]\big)
= \widehat{F}\big([X_1\times X_2, \mathfrak{l}_1\times \mathfrak{l}_2]\big)
= \big[C\times_{\mathfrak{l}_2\times\mathfrak{l}_2}(X_1\times X_2)\big]
=\big[(C\times_{\mathfrak{l}_1}X_1)\times_C(C\times_{\mathfrak{l}_2}X_2)\big].$$
Thus, the desired result follows.

\end{proof}
\begin{rem}\label{lx}
Let $(X,\mathfrak{l})$ be a $C$-monomial $G$-set.
For all $x \in X$, we get a character 
$\mathfrak{l}_x : G_x \rightarrow C$ 
defined by $\mathfrak{l}_x(g)= \mathfrak{l}(g, x,x)$ 
for $g \in G_x$.  On the other hand given a subgroup $U$
of $G$ and a group homomorphism $\mu : U \rightarrow C$ we 
get a $C$-monomial $G$-set $(G/U, {\rouge \widehat{\mu}})$ where 
and ${\rouge \widehat{\mu}}:  \widehat{G/U} \rightarrow \bullet_{C}$ 
is the functor such that given $gU,\,kU \in G/U$ if $hgU=kU$
for some $g \in G$ then ${\rouge \widehat{\mu}}(h, gU, kU)= \mu(k^{-1}hg)$. {\rouge Moreover,
$[U,\mu]_G$ and $[G/U, \widehat{\mu}]$ represents the same element in $B_C(G)$.}
\end{rem}

\subsection{The Lefschetz invariant attached to a monomial $G$-poset}
A $G$-poset $X$ is a partially ordered set $(X,\leq)$ with a compatible $G$-action (that is $gx\leq gy$ whenever $g\in G$ and $x\leq y$ in $X$). A map of $G$-posets is a $G$-equivariant map of posets. We denote by $G\hbox{-\sf poset}$ the category of finite $G$-posets obtained in this way.\par
There is an obvious functor $\iota_G: G\hbox{-\sf set}\to G\hbox{-\sf poset}$ sending each finite $G$-set to the set $X$ ordered by the equality relation, and each $G$-equivariant map to itself.\par
The Lefschetz invariant attached to a finite $G$-poset,
which is an element of the Burnside ring of~$G$ has been
introduced in \cite{LEF} by Th\'evenaz.
We will define similarly a Lefschetz invariant attached
to a $C$-monomial $G$-poset as an element of the
$C$-monomial Burnside ring of $G$.\par

\subsubsection{The category of $C$-monomial $G$-posets}

Given a $G$-poset $X$, we consider the category
 $\widehat{X}$  whose objects 
are the elements of $X$ and given $x$, 
$y$ in $X$ the set of morphisms from $x$ to $y$ is 
$$\Hom_{\widehat{X}}(x, y)= \{ g \in G \mid gx \leq y\}.$$
Now we define a $C$-monomial $G$-poset as follows.

\begin{defn}
	A {\em$C$-monomial $G$-poset} is a pair $(X, \mathfrak{l})$ consisting of a $G$-poset $X$ and a functor $\mathfrak{l}: \widehat{X} \rightarrow \bullet_{C}$.   
\end{defn}
{\dgreen In otherwords, for each $x,y\in X$ and $g\in G$ such that $gx\leq y$, we have an element $\mathfrak{l}(g,x,y)$ of $C$, with the property that $\mathfrak{l}(h,y,z)\mathfrak{l}(g,x,y)=\mathfrak{l}(hg,x,z)$ if $h\in G$ and $hy\leq z$, and $\mathfrak{l}(1,x,x)=1$ for any $x\in X$.}\par
Let $(X, \mathfrak{l})$ and $(Y, \mathfrak{m})$ be $C$-monomial 
$G$-posets. {\em A map of $C$-monomial $G$-posets} from $(X, \mathfrak{l})$ to $(Y, \mathfrak{m})$ is a pair
$(f, \lambda) : (X, \mathfrak{l}) \rightarrow (Y, \mathfrak{m})$,
where $f: X \rightarrow Y $ is a map of $G$-posets and 
{\rouge $\lambda : \mathfrak{l} \rightarrow \mathfrak{m} \circ f$} is
a natural transformation. We denote the
category of $C$-monomial $G$-posets by $_CMG$-{\sf poset}. 
Product and disjoint union of $C$-monomial $G$-posets
	are defined as for $C$-monomial $G$-sets.
When $C$ is the trivial group, we will identify the category $_CMG\hbox{-\sf poset}$ with $G\hbox{-\sf poset}$.\par
\begin{rem}\label{res}
If $(X,\mathfrak{l})$ is a $C$-monomial $G$-poset, then for any $x \in X$
we get a character $\mathfrak{l}_x: G_x \rightarrow C$ defined by 
$\mathfrak{l}_x(g)=\mathfrak{l}(g,x,x)$. Moreover, if $x\leq y$, then
$$\ssres^{G_x}_{G_x\cap{G_y}}\mathfrak{l}_x= \ssres^{G_y}_{G_x\cap{G_y}}\mathfrak{l}_y$$
because we have the following commutative diagram:
{\dgreen
$$\xymatrix@C=7ex@R=7ex{
\mathfrak{l}(x)\ar[r]^-{\mathfrak{l}(1, x, y)}\ar[d]_-{\mathfrak{l}(g, x, x)}&\mathfrak{l}(y)\ar[d]^-{\mathfrak{l}(g, y, y)}\\
\mathfrak{l}(x)\ar[r]_-{\mathfrak{l}(1, x, y)}&\mathfrak{l}(y).
}
$$
}
%
%
%

\end{rem}

Let $H$ be a subgroup of $G$ { \rouge and $(X, \mathfrak{l})$ be a $C$-monomial $H$-set.} 
We let $G\times_HX$ to be the  quotient of $G\times X$ by the action of $H$.
The set $G\times_HX$ is a $G$-set 
via the action $g(u,_{_H} x) = (gu,_{_H} x)$, for any 
$g\in G$, and $ (u,_{_H} x) \in G\times_H X$.
We define an order relation
$\leq$ on $G\times_HX$ as
$$\forall  (u,_{_H}x),\,(v,_{_H}y)\in G\times_HX,\,\,
(u,_{_H}x)\leq (v,_{_H}y) \Leftrightarrow \exists h \in H,\, u=vh,\,x\leq h^{-1}y.$$
Since we have 
$$G\times_HX= \bigsqcup\limits_{g \in G/H} g\times_HX,$$
it's enough to consider the chains of type 
$(u,_{_H}x_0)<...<(u,_{_H}x_n)$ in $G\times_HX$ for some $u\in G$ and a chain
$x_0<...<x_n$ in $X$ for some $n\in \NN$.\par
{\rouge Let $(u,_{_H}x),\,(u,_{_H}y)\in G\times_HX$ and $g \in G$ such that $g(u,_{_H}x)\leq (u,_{_H}y)$.}
 Then
there exists $h\in H$ such that 
$gu=uh$ and $hx\leq y$. 
We define {\em the induced $C$-monomial $G$-poset} $\Ind^G_H(X,\mathfrak{l})$  of $(X,\mathfrak{l})$ as the pair $(G\times_HX, G\times_H\mathfrak{l})$
where $G\times_H\mathfrak{l} : \widehat{G\times_HX} \rightarrow \bullet_C$ 
is defined by
$$(G\times_H\mathfrak{l})\big(g,(u,_{_H}x),(u,_{_H}y)\big)
=\mathfrak{l}(h,x,y).$$
Now show that  $(G\times_HX, G\times_H\mathfrak{l})$ is a
$C$-monomial $G$-poset.\par
Let $(u,_{_H}x),\,(u,_{_H}y),\,(u,_{_H}z) \in G\times_HX$
such that 
$$g(u,_{_H}x)\leq (u,_{_H}y)$$  and 
$$g^\prime(u,_{_H}y)\leq (u,_{_H}z)$$
for some $g,\,g^\prime \in G$. Then there exist some $h,\,h^\prime \in H$
such that 
$$ gu=uh,\,\,\,g^\prime u=uh^\prime,\,\,\, hx\leq y,\,\,\,h^\prime y\leq z.$$
Then $t=h'h\in H$. Moreover $g^\prime gu=uh^\prime h=ut$ and $tx=h'hx\leq z$. 
Now we get
$$
(G\times_H\mathfrak{l})\big(g^\prime g, (u,_{_H}x), (u,_{_H}z) \big)
= \mathfrak{l}(t,x,z)=\mathfrak{l}(h^\prime h,x,z)
= \mathfrak{l}(h^\prime,y,z)\mathfrak{l}(h,x,y)$$
$$= (G\times_H\mathfrak{l})\big(g^\prime, (u,_{_H}x), (u,_{_H}y) \big)
(G\times_H\mathfrak{l})\big(g, (u,_{_H}y), (u,_{_H}z) \big).
$$
We also have $(G\times_H\mathfrak{l})\big(1,(u,_{H}x),(u,_{_H}x)\big)=1$
for any $(u,_{_H}x) \in G\times_HX$. Thus $G\times_H\mathfrak{l}$ is a functor.
So  $\Ind^G_H(X,\mathfrak{l})$  is a $C$-monomial $G$-poset.\par
Given a $C$-monomial $G$-poset $(Y,\mathfrak{m})$, {\em the restriction} $\Res_H^G(Y,\mathfrak{m})$ of $(Y,\mathfrak{m})$
is the pair $(\Res^G_HY, \ssres^G_H\mathfrak{m})$
where $\Res^G_HY$ is the restriction of the $G$-poset $Y$ to $H$-poset and
$\ssres^G_H\mathfrak{m}$ is the restriction of the functor $\mathfrak{m}$ from 
$\widehat{Y}$ to $\widehat{\Res^G_HY}$.
\begin{prop} Let $G$ be a finite group.
\begin{enumerate}
\item If $Y$ is a finite $G$-poset, denote by $1_Y:\widehat{Y}\to \bullet_C$ the trivial functor defined by $1_Y(g,x,y)=1$ for any $g\in G$ and $x,y\in Y$ such that $gx\leq y$. Then the assignment $Y\mapsto (Y,1_Y)$ is a functor $\tau_G$ from $G\hbox{-\sf poset}$ to $_CMG\hbox{-\sf poset}$.
\item Let $H$ be a subgroup of $G$. The assignment $(X,\mathfrak{l})\mapsto \Ind_H^G(X,\mathfrak{l})$ is a functor $\Ind_H^G:{_CMH\hbox{-\sf poset}}\to{_CMG\hbox{-\sf poset}}$, and the assignment $(Y,\mathfrak{m})\mapsto \Res_H^G(Y,\mathfrak{m})$ is a functor $\Res_H^G:{_CMG\hbox{-\sf poset}}\to{_CMH\hbox{-\sf poset}}$.
\item Moreover the diagrams 
$$\xymatrix{
H\hbox{-\sf poset}\ar[r]^-{\Ind_H^G}\ar[d]_{\tau_H}&G\hbox{-\sf poset}\ar[d]_{\tau_G}\\
_CMH\hbox{-\sf poset}\ar[r]^-{\Ind_H^G}&_CMG\hbox{-\sf poset}
}\hbox{and}\xymatrix{
G\hbox{-\sf poset}\ar[r]^-{\Res_H^G}\ar[d]_{\tau_G}&H\hbox{-\sf poset}\ar[d]_{\tau_H}\\
_CMG\hbox{-\sf poset}\ar[r]^-{\Res_H^G}&_CMH\hbox{-\sf poset}
}$$
of categories and functors are commutative.
\end{enumerate}
\end{prop}
\begin{proof}
\begin{enumerate}
	\item Let $f: X \rightarrow Y$ be a map of $G$-posets. We set 
	$$\tau_G(f)= (f,1_f): (X,1_X) \rightarrow (Y,1_Y),$$
	where $1_f: 1_X \rightarrow 1_Y\circ f$ is defined by
	${1_f}_x= 1$ for any $x \in X$. Obviously $(f,1_f)$ is a map of $C$-monomial
	$G$-posets and $\tau_G$ is a functor.
	\item Let $(f,\lambda):(X,\mathfrak{l}) \rightarrow (Y,\mathfrak{m})$ be a map of $C$-monomial $H$-posets. 
	 We set the pair 
	 $$\Ind^G_H(f,\lambda)=(G\times_Hf,G\times_H\lambda): (G\times_HX, G\times_H\mathfrak{l}) \rightarrow (G\times_HY, G\times_H\mathfrak{m})$$
	 where 
	 $$G\times_Hf : G\times_HX \rightarrow G\times_HY$$
	 is defined by
	 $(G\times_Hf)(u,_{_H}x)= \big(u,_{_H}f(x)\big)$ and
	 $$G\times_H\lambda :  G\times_H\mathfrak{l} 
	 \rightarrow (G\times_H\mathfrak{m})\circ(G\times_Hf)$$ 
	 is defined by $(G\times_H\lambda)_{(u,_{_H}x)}= \lambda_x$ for any $(u,_{_H}x) \in G\times_HX$. 
	 It's clear that $G\times_Hf$ is a map of $C$-monomial $G$-posets. Now we show that 
	 $G\times_H\lambda$ is a natural transformation. Let $(u,_{_H}x),\,(u,_{_H}y) \in G\times_HX$
	 such that $g(u,_{_H}x)\leq (u,_{_H}y)$ for some $g\in G.$ Then $gu=uh$ and $hx\leq y$
	 for some $h \in H$. Since $\lambda: \mathfrak{l} \rightarrow \mathfrak{m}\circ f$ is
	 a natural transformation, we get
	 $$ (G\times_H\mathfrak{m})\Big(g,\big(u,_{_H}f(x)\big),\big(u,_{_H}f(y)\big)\Big)
	 (G\times_H\lambda)_{(u,_{_H}x)}= \mathfrak{m}\big(h,f(x),f(y)\big)\lambda_x$$
	 $$=\lambda_y\mathfrak{l}(h,x,y)= 
	 {\rouge (G\times_H\lambda)_{(u,_{_H}y)}}(G\times_H\mathfrak{l})\big(g,(u,_{_H}x),(u,_{_H}y)\big).$$
	 Now consider $(\id_X,\id_{\mathfrak{l}}): (X,\mathfrak{l}) \rightarrow (X,\mathfrak{l})$
	 where $\id_X: X \rightarrow X$ is the identity map on the $H$-set $X$ and
	 $\id_{\mathfrak{l}}:\mathfrak{l} \rightarrow \mathfrak{l}\circ\id_X$ is the identity transformation.
	 Then we get
	 $$\Ind^G_H(\id_X,\id_{\mathfrak{l}})= (\id_{G\times_HX},\id_{G\times_H\mathfrak{l}}).$$
	 Now let $(f,\lambda): (X,\mathfrak{l}) \rightarrow (Y,\mathfrak{m})$ and
	 $(t,\beta): (Y,\mathfrak{m})\rightarrow (Z,\mathfrak{r})$ be the maps of $C$-monomial $H$-posets. We obviously have 
	 $$(G\times_Ht)\circ(G\times_Hf)= G\times_H(t\circ f)$$ 
	 and 
	 $$(G\times_H\beta)\circ(G\times_H\lambda)= G\times_H(\beta\circ\lambda).$$
	 Thus, 
	 $$\Ind^G_H(t,\beta)\circ\Ind^G_H(f,\lambda)= \Ind^G_H\big((t,\beta)\circ(f,\lambda)\big).$$ 
	 So $\Ind^G_H:{_CMH\hbox{-\sf poset}}\to{_CMG\hbox{-\sf poset}}$ is a functor.\par
	 Now let $(f,\lambda): (X,\mathfrak{l})\rightarrow (Y,\mathfrak{m})$ be 
	 a map of $C$-monomial $G$-posets. We set the pair
	 $$\Res^G_H(f,\lambda)=(f|_H,\lambda|_H): (\Res^G_HX,\ssres^G_H\mathfrak{l})
	 \rightarrow (\Res^G_HY,\ssres^G_H\mathfrak{m})$$
	 where $f|_H: \Res^G_HX \rightarrow \Res^G_HY$ is defined as the restriction of map of $G$-posets $f$ to map of $H$-posets and 
	 $\lambda|_H : \ssres^G_H\mathfrak{l} \rightarrow \ssres^G_H\mathfrak{m}\circ f|_H$
	 is defined as the restriction of $\lambda$. Clearly, we get that
	 $\Res_H^G:{_CMG\hbox{-\sf poset}}\to{_CMH\hbox{-\sf poset}}$ is a functor.
	 \item Let $X$ be an $H$-poset. Commmutativity of the first diagram follows from
	 $$\tau_G\circ\Ind^G_H(X)= \tau_G(G\times_HX)= (G\times_HX,1_{G\times_HX})
	 =\Ind^G_H(X,1_X)=\Ind^G_H\circ\tau_H(X).$$
	 Now let $Y$ be a $G$-poset. Commutativity of the second diagram follows from
	 $$\tau_H\circ\Res^G_H(Y)=\tau_H(\Res^G_HY)=(\Res^G_HY,1_{\Res^G_HY})$$
	 $$=(\Res^G_HY,\ssres^G_H1_Y)=\Res^G_H(Y,1_Y)=\Res^G_H\circ\tau_H(Y).$$
	 
\end{enumerate}
\end{proof}
\begin{prop} Let $G$ be a finite group and $H$ be subgroup of $G$. Then the functor $\Ind_H^G:{_CMH\hbox{-\sf poset}}\to{_CMG\hbox{-\sf poset}}$ is left adjoint to the functor $(Y,\mathfrak{m})\mapsto \Res_H^G(Y,\mathfrak{m})$.
\end{prop}
\begin{proof} We prove that for any $C$-monomial $H$-poset $(X,\mathfrak{l})$ and any $C$-monomial $G$-poset $(Y,\mathfrak{m})$ we have a bijection
	 $$\Hom_{_CMG}\big(\Ind^G_H(X,\mathfrak{l}),(Y,\mathfrak{m})\big)
		\cong \Hom_{_CMH}\big((X,\mathfrak{l}), \Res^G_H(Y,\mathfrak{m})\big)$$
natural in $(X,\mathfrak{l})$ and $(Y,\mathfrak{m})$.  

	We define 
	$$\varphi: \Hom_{_CMG}\big(\Ind^G_H(X,\mathfrak{l}),(Y,\mathfrak{m})\big) \rightarrow \Hom_{_CMH}\big((X,\mathfrak{l}), \Res^G_H(Y,\mathfrak{m})\big)$$

   where
	$$\varphi: (f,\lambda) \mapsto \big(\varphi(f),\varphi(\lambda)\big)$$
	such that 
	$$\varphi(f): X \rightarrow \Res^G_H(Y)$$ 
	defined by
	$\varphi(f)(x)=f(1,_{_H}x)$ and 
	$$\varphi(\lambda):\mathfrak{l} \rightarrow \ssres\mathfrak{m}\circ \varphi(f)$$
	defined by 
	$\varphi(\lambda)_x=\lambda_{(1,_{_H}x)}$
	for any $x \in X$. Obviously, $\varphi(f)$ is a map of
	$H$-posets. We need to show that 
	$$\varphi(\lambda): \mathfrak{l} \rightarrow 
	\ssres\mathfrak{m}\circ \varphi(f)$$
	is a natural transformation. 
	Let $x,\,y \in X$ such that $gx\leq y$ for some $g \in G$. Then
	
	\begin{align*}
	\mathfrak{m}\big(h,&\varphi(f)(x), \varphi(f)(y)\big)\varphi(\lambda)_x
	=\mathfrak{m}\big(h, f(1,_{_H}x), f(1,_{_H}y)\big)
	\lambda_{(1,_{_H}x)}\\
	&= \lambda_{(1,_{_H}y)}\mathfrak{l}(h,x,y)= 
	\varphi(\lambda)_{y}
	(G\times_H\mathfrak{l})\big(h,(1,_{_H}x),(1,_{_H}y)\big).\\
	\end{align*}
	We define an inverse map to $\varphi$ as
		$$\theta: \Hom_{_CMH}\big((X,\mathfrak{l}), \Res^G_H(Y,\mathfrak{m})\big) 
	\rightarrow \Hom_{_CMG}\big(\Ind^G_H(X,\mathfrak{l}),(Y,\mathfrak{m})\big)$$
	where 
	 $$\theta : (\psi, \beta) \mapsto  \big(\theta(\psi),{\rouge\theta(\beta)}\big)$$
	such that 
	{\magenta $$\theta(\psi): G\times_HX \rightarrow Y$$
	defined as $\theta(\psi)(u,_{_H}x)=u\psi(x)$ and
        $$\theta(\beta): G\times_H\mathfrak{l}\rightarrow \mathfrak{m}\circ\theta(\psi)$$
    defined as }
	$$\theta(\beta)_{(u,_{_H}x)}=
	\mathfrak{m}\big(u,\psi(x),u\psi(x)\big)\beta_x$$
	for any $(u,_{_H}x) \in G\times_HX.$ Obviously, the map $\theta(\psi)$
	is {a map of $G$-posets}. We need to show that {\rouge $\theta(\beta)$} is
	a natural transformation. Let 
	$(u,_{_H}x),\,(u,_{_H}y)\in G\times_HX$ such that 
	$g(u,_{_H}x)\leq (u,_{_H}y)$ for some $g\in G$. 
	Then there exists some $h \in H$ such that 
	{\rouge $gu=uh$ and $hx\leq y$.} Now, we have
	\begin{align*}
	\mathfrak{m}&\big(g, \theta(\psi)(u,_{_H}x), \theta(\psi)(u,_{_H}y)\big){\rouge\theta(\beta)_{(u,_{_H}x)}}
     = \mathfrak{m}\big(g,u\psi(x),u\psi(y)\big)
	\mathfrak{m}\big(u,\psi(x),u\psi(x)\big)\beta_x\\
	&= \mathfrak{m}\big(u,\psi(y),u\psi(y)\big)
	\mathfrak{m}\big(h,\psi(x),\psi(y)\big) \beta_x
	= {\rouge \mathfrak{m}\big(u,\psi(y),u\psi(y)\big)}\beta_{y}
	\mathfrak{l}(h,x,y)\\
	&{\rouge =\theta(\beta)_{(u,_{_H}y)}(G\times_H\mathfrak{l})\big(h,(u_{_H}x),(u,_{_H}y)\big).}\\
	\end{align*}
  Clearly, $\varphi$ and $\theta$ are mutual inverse maps, and natural in $(X,\mathfrak{l})$ and $(Y,\mathfrak{m})$.	
	
\end{proof}

\subsubsection{The Lefschetz invariant attached to a $C$-monomial $G$-poset}
Let $(X,  \mathfrak{l})$ be a $C$-monomial $G$-poset. 
The {\em Lefschetz invariant} $\Lambda_{(X,  \mathfrak{l})}$ of $(X,  \mathfrak{l})$ 
 is the element of $B_{C}(G)$ defined by

$$\Lambda_{(X,  \mathfrak{l})} = 
\sum\limits_{x_0 < ... < x_n \in_G X}
(-1)^n \big[{\rouge G_{x_0, ..., x_n}}, \Res^{G_{x_0}}_{G_{x_0, ..., x_n}}( \mathfrak{l}_{x_0})\big]_G$$
where $x_0 < ... < x_n$ runs over $G$-representatives of the 
chains in $X$. {\dgreen The group }{\rouge $G_{x_0, ..., x_n}$ is the stabilizer of the set $\{x_0, ..., x_n\}$,} {\dgreen that is} {\rouge $G_{x_0, ..., x_n}= \cap_{i=0}^nG_{x_i}$.} {\dgreen Here $\Res^{G_{x_0}}_{G_{x_0, ..., x_n}}( \mathfrak{l}_{x_0})$ denotes the restriction of the character $\mathfrak{l}_{x_0}$ introduced in Remark~\ref{lx}. Observe that i}f $x_0 < ... < x_n$  is a chain in $X$ for some $n\in \NN$, by Remark \ref{res}
we have 
$$\Res^{G_{x_0}}_{G_{x_0, ..., x_n}}\mathfrak{l}_{x_0}=
\Res^{G_{x_i}}_{G_{x_0, ..., x_n}}\mathfrak{l}_{x_i}$$
for any $0\leq i \leq n $.\par
Let $(X,\mathfrak{l})$ be a $C$-monomial $G$-poset. Given $n\in \NN$, let $\Sd_n(X)$
denote the set of chains in $X$ with order $n+1$. Obviously, the set $\Sd_n(X)$
is a $G$-set. Then $(\Sd_n(X),\mathfrak{l}_n)$ is a $C$-monomial $G$-set where 
$\mathfrak{l}_n: \widehat{\Sd_n(X)} \rightarrow \bullet_C$ is the functor defined by
$$\mathfrak{l}_n(g, x_0<...<x_n, y_0<...<y_n)= \mathfrak{l}(g,x_0,y_0)$$
for any $x_0<...<x_n$, and $y_0<...<y_n$ in $\Sd_n(X)$ such that
$$g(x_0<...<x_n)=  y_0<...<y_n$$
for some $g\in G$.
\begin{rem}\label{chn}Given a $C$-monomial $G$-{\dgreen poset} $(X,\mathfrak{l})$, {\dgreen
	we have the following isomorphism of monomial $G$-sets:}
	$$\big(\Sd_n(X),\mathfrak{l}_n\big)\cong
	\bigsqcup\limits_{x_0<...<x_n\in_G \Sd_n(X)}
	\big(G/G_{x_0,...,x_n},{\dgreen \widehat{\Res^{G_{x_0}}_{G_{x_0,...,x_n}}(\mathfrak{l}_{x_0})}}\big)$$
	for any $n \in \NN$.
\end{rem}
\begin{proof}
	Let $[G/\Sd_n(X)]$ be a set of representative of the $G$-action on $\Sd_n(X)$. 
	Let $x=x_0<...<x_n$ be a chain in $\Sd_n(X)$ then there exist some $g_x\in G$ and a unique 
	$\sigma_{x} \in [G/\Sd_n(X)]$ such that $x=g_x\sigma_{x}$ where 
	$\sigma_x= \sigma_{x_0}<...<\sigma_{x_n}$. 
	We define 
	$$(f,\lambda): \big(\Sd_n(X),\mathfrak{l}_n\big) 
	\rightarrow \bigsqcup\limits_{x_0<...<x_n\in_G\Sd_n(X)}
	\big(G/G_{x_0,...,x_n},{\dgreen \widehat{\Res^{G_{x_0}}_{G_{x_0,...,x_n}}(\mathfrak{l}_{x_0})}}\big)$$
	where {\rouge $f(x)=g_xG_{\sigma_x} \in G/G_{\sigma_x}$} and $\lambda_x= \mathfrak{l}(g_x^{-1},g_x\sigma_{x_0}, \sigma_{x_0})$. 
	Obviously, 
	$$f: \Sd_n(X)\rightarrow \bigsqcup\limits_{x_0<...<x_n\in_G\Sd_n(X)}G/G_{x_0,...,x_n}$$ 
	is an isomorphism of $G$-sets. We show that 
	$$\lambda: \mathfrak{l}_n \rightarrow \!\!\!\!\!\!\bigsqcup\limits_{x_0<...<x_n\in_GSd_n(X)}
	\!\!\!\!\!\!
	{\dgreen \widehat{\Res^{G_{x_0}}_{G_{x_0,...,x_n}}(\mathfrak{l}_{x_0})}}\circ f$$
	is a natural transformation. 	Let $x= x_0<...<x_n$, and $y=  y_0<...<y_n$ be sequences 
	in $\Sd_n(X)$
	such that $gx=y$ for some $g \in G$. There exist a unique 
	$\sigma_x,\sigma_y \in [G/\Sd_n(X)]$ such that $x=g_x\sigma_x$ and $y=g_y\sigma_y$
	for some $g_x$ and $g_y$ in $G$. Then $x_0=g_x\sigma_{x_0}$ and $y_0=g_y\sigma_{y_0}$
	so $y_0=gx_0= gg_x\sigma_{x_0}$. Thus, by uniqueness $\sigma_{x_0}=\sigma_{y_0}$ and so
	$g_y^{-1}gg_x \in G_{\sigma_{x_0}}$. Then setting $r={\dgreen \widehat{\Res^{G_{x_0}}_{G_x}(\mathfrak{l}_{x_0})}}\big(g,f(x),f(y)\big)\lambda_x$, we have that

	\begin{align*}
		r&= \widehat{\mathfrak{l}_{x_0}}(g,g_xG_{\sigma_{x}}, g_yG_{\sigma_{y}})\mathfrak{l}(g_x^{-1},g_x\sigma_{x_0}, \sigma_{x_0})
	= \mathfrak{l}_{x_0}(g_y^{-1}gg_x)\mathfrak{l}(g_x^{-1},g_x\sigma_{x_0}, \sigma_{x_0})\\
	&=\mathfrak{l}(g_y^{-1}gg_x,x_0,x_0) \mathfrak{l}(g_x^{-1},g_x\sigma_{x_0}, \sigma_{x_0})\\
	&=\mathfrak{l}(g_x,\sigma_{x_0},g_x\sigma_{x_0})\mathfrak{l}(g,g_x\sigma_{x_0},gg_x\sigma_{x_0})
	\mathfrak{l}(g_y^{-1},gg_x\sigma_{x_0},\sigma_{x_0})
	\mathfrak{l}(g_x^{-1},g_x\sigma_{x_0}, \sigma_{x_0})\\
	&=\mathfrak{l}(g,x_0,y_0)\mathfrak{l}(g_y^{-1},g_y\sigma_{y_0},\sigma_{y_0})\\
	&= \mathfrak{l}_n(g,x,y)\lambda_y.\\
	\end{align*}
\end{proof}

By Remark \ref{chn}, the Lefschetz invariant of a $C$-monomial $G$-set 
$(X,\mathfrak{l})$ can be written~as
$$\Lambda_{(X,\mathfrak{l})}= \sum\limits_{x_0 < ... < x_n \in_G X}
(-1)^n \big[G_{x_0, ... x_n}, \Res^{G_{x_0}}_{G_{x_0, ..., x_n}}( \mathfrak{l}_{x_0})\big]_G
= \sum\limits_{n \in \NN}(-1)^n \big(\Sd_n(X),\mathfrak{l}_n\big).$$
 It follows that $\Lambda_X=\Lambda_{\tau_G(X)}$, where $\Lambda_X$ the Lefschetz invariant of the $G$-poset $X$ introduced in \cite{BIS}.\par
We define similarly the {\em reduced Lefschetz invariant} of $(X,\mathfrak{l})$
$$\widetilde{\Lambda}_{(X,\mathfrak{l})}= \Lambda_{(X,  \mathfrak{l})}-[G,1_G]_G$$
where $1_G$ is the trivial character of $G$.

\begin{lemma}\label{lefprp}Let $G$ be a finite group and $C$ be an abelian group.
	\begin{enumerate}
		\item Let $(X,\mathfrak{l})$ be a $C$-monomial $G$-set, viewed a a $C$-monomial $G$-poset ordered by the equality relation on $X$. Then 
		   $\Lambda_{(X, \mathfrak{l})}= [C\times_{\mathfrak{l}}X]$ in $B_C(G)$.
		
		\item Let  $(X, \mathfrak{l})$ and $(Y, \mathfrak{m})$
		 be $C$-monomial $G$-posets. 
		Then $\Lambda_{(X\sqcup Y, \mathfrak{r})}= 
		\Lambda_{(X, \mathfrak{l})}+ \Lambda_{(Y, \mathfrak{m})}$
		 in $B_C(G)$.
		
		\item Given $C$-monomial $G$-posets $(X, \mathfrak{l})$
		 and $(Y, \mathfrak{m})$,
		 we have $\Lambda_{(X \times Y, \mathfrak{l} \times \mathfrak{m})}
		 = \Lambda_{(X, \mathfrak{l})}\Lambda_{(Y, \mathfrak{m})}$ in $B_C(G).$
		
	\end{enumerate}
	
\end{lemma}

\begin{proof} $1.$ and $2.$ are clear.
	
	$3$. In the following proof using the inclusion
	$$B_C(G) \hookrightarrow \QQ\otimes_{\ZZ}B_C(G)$$
	 we identify the elements of $B_C(G)$ with their image
	 in $\QQ\otimes_{\ZZ}B_C(G)$. We start with rearranging 
	 the chains in $X\times Y$ as in the proof of 
	 Lemma 11.2.9 in \cite{BIS}.
	Let $n \in \NN$. Given a chain $z=z_0< ... < z_n$
	 in $X \times Y$ projection of $z$ on $X$ is 
	denoted by $z_X$ and on $Y$ is denoted by $z_Y$. 
	Then $z_X$ is a chain in $X$ with order $i+1$
	 for some $i \leq n$ and $z_Y$ is a chain in $Y$ with 
	 order $j+1$ for some $j \leq n$ {\rouge such that $i+j=n$.} Let $\underline{s}_i$
	  be the chain $s_0< ... < s_i$ and $\underline{t}_j$ 
	  be the chain $t_0< ... < t_j$. Now
	
	\begin{align*}
	\Lambda_{(X \times Y, \mathfrak{l} \times \mathfrak{m})} &= \sum\limits_{\substack{n \in \NN, \\ {\rouge z \in_G \Sd_n(X\times Y)}}}(-1)^n \big[G_z, \Res^{G_{z_0}}_{G_z}(\mathfrak{l}_{z_0})\big]_G= \sum\limits_{\substack{n \in \NN, \\ {\rouge z \in \Sd_n(X\times Y)}}}(-1)^n \frac{|G_z|}{|G|}[G_{z}, \Res^{G_{z_0}}_{G_z}(\mathfrak{l}_{z_0})]_G\\
	&=\sum\limits_{\substack{i,j \in \NN \\ \underline{s}_i \in X \\ \underline{t}_j \in Y }}\Gamma_{\underline{s}_i,\underline{t}_j} \\
	 \end{align*}

	where 
	\begin{align*}
	\Gamma_{\underline{s}_i,\underline{t}_j} &=\!\!\!\!\!\!\!\!\!\!\!\sum\limits_{\substack{ n \in \NN \\ {\rouge z \in \Sd_n(X \times Y)}:z_X=\underline{s}_i,\,z_Y=\underline{t}_j}}\!\!\!\!\!\!\!\!\!\!\!(-1)^n \frac{|G_{\underline{s}_i}\cap G_{\underline{t}_j}|}{|G|}\big[G_{\underline{s}_i}\cap G_{\underline{t}_j}, \Res^{G_{s_0}}_{G_{\underline{s}_i}}(\mathfrak{l}_{s_0}) \Res^{G_{t_0}}_{G_{\underline{t}_j}}(\mathfrak{m}_{t_0})\big]_G\\
	 &= \frac{|G_{\underline{s}_i}\cap G_{\underline{t}_j}|}{|G|}\big[G_{\underline{s}_i}\cap G_{\underline{t}_j}, \Res^{G_{s_0}}_{G_{\underline{s}_i}}(\mathfrak{l}_{s_0}) \Res^{G_{t_0}}_{G_{\underline{t}_j}}(\mathfrak{m}_{t_0})\big]_G\!\!\!\!\!\!\!\!\!\!\!\sum\limits_{\substack{ n \in \NN \\ {\rouge z \in \Sd_n(X \times Y)}:z_X=\underline{s}_i,\,z_Y=\underline{t}_j}}\!\!\!\!\!\!\!\!\!\!\!(-1)^n\\
	 &= \frac{|G_{\underline{s}_i}\cap G_{\underline{t}_j}|}{|G|}\big[G_{\underline{s}_i}\cap G_{\underline{t}_j}, \Res^{G_{s_0}}_{G_{\underline{s}_i}}(\mathfrak{l}_{s_0}) \Res^{G_{t_0}}_{G_{\underline{t}_j}}(\mathfrak{m}_{t_0})\big]_G(-1)^{i+j}.
	\end{align*}
	 
	Now,
	 $$\Lambda_{(X \times Y, \mathfrak{l} \times \mathfrak{m})} =\sum\limits_{\substack{i,j \in \NN \\ \underline{s}_i \in X \\ \underline{t}_j \in Y}} (-1)^{i+j}\frac{|G_{\underline{s}_i}\cap G_{\underline{t}_j}|}{|G|}\big[G_{\underline{s}_i}\cap G_{\underline{t}_j}, \Res^{G_{s_0}}_{G_{\underline{s}_i}}(\mathfrak{l}_{s_0}) \Res^{G_{t_0}}_{G_{\underline{t}_j}}(\mathfrak{m}_{t_0}) \big]_G.$$
		
 On the other hand 
 
 \begin{align*}
 \Lambda_{(X, \mathfrak{l})}\Lambda_{(Y, \mathfrak{m})} &= \sum\limits_{\substack{i \in \NN \\ \underline{s}_i \in_G X }}(-1)
 \big[G_{\underline{s}_i},\Res^{G_{s_0}}_{G_{\underline{s}_i}}(\mathfrak{l}_{s_0})\big]\sum\limits_{\substack{j \in \NN \\ \underline{t}_j \in_G Y }}(-1)^j\big[G_{\underline{t}_j},\Res^{G_{t_0}}_{G_{\underline{t}_j}}(\mathfrak{m}_{t_0})\big]_G\\
 &= \!\!\!\!\sum\limits_{\substack{i,j \in \NN \\ \underline{s}_i \in X \\  \underline{t}_j \in Y \\ G_{^i}gG_{\underline{t}_j} \subseteq G}}\!\!\!\! (-1)^{i+j}\frac{|G_{\underline{s}_i}||G_{\underline{t}_j}|}{|G|^2}\big[G_{\underline{s}_i}\cap {^gG_{\underline{t}_j}}, \Res^{G_{s_0}}_{G_{\underline{s}_i}}(\mathfrak{l}_{s_0}) \Res^{^gG_{t_0}}_{^gG_{\underline{t}_j}}(^g\mathfrak{m}_{t_0})\big]_G\\
 &= \sum\limits_{\substack{i,j \in \NN \\ \underline{s}_i \in X \\  \underline{t}_j \in Y \\ g \in G}} (-1)^{i+j}\frac{|G_{\underline{s}_i}||G_{\underline{t}_j}|}{|G|^2|G_{\underline{s}_i}gG_{\underline{t}_j}|}\big[G_{\underline{s}_i}\cap {^gG_{\underline{t}_j}}, \Res^{G_{s_0}}_{G_{\underline{s}_i}}(\mathfrak{l}_{s_0}) \Res^{^gG_{t_0}}_{^gG_{\underline{t}_j}}(^g\mathfrak{m}_{t_0}) \big]_G\\
 &= \sum\limits_{\substack{i,j \in \NN \\ \underline{s}_i \in X \\  \underline{t}_j \in Y \\ g \in G}} (-1)^{i+j}\frac{|G_{\underline{s}_i}\cap {^gG_{\underline{t}_j}}|}{|G|^2}\big[G_{\underline{s}_i}\cap {^gG_{\underline{t}_j}}, \Res^{G_{s_0}}_{G_{\underline{s}_i}}(\mathfrak{l}_{s_0}) \Res^{^gG_{t_0}}_{^gG_{\underline{t}_j}}(^g\mathfrak{m}_{t_0}) \big]_G\\
 &= \sum\limits_{\substack{i,j \in \NN \\ \underline{s}_i \in X \\  \underline{t}_j \in Y \\ g \in G}} (-1)^{i+j}\frac{|G_{\underline{s}_i}\cap {G_{g\underline{t}_j}}|}{|G|^2}\big[G_{\underline{s}_i}\cap {G_{g\underline{t}_j}}, \Res^{G_{s_0}}_{G_{\underline{s}_i}}(\mathfrak{l}_{s_0}) \Res^{G_{gt_0}}_{G_{g\underline{t}_j}}(\mathfrak{m}_{gt_0}) \big]_G\\
  &= \sum\limits_{\substack{i,j \in \NN \\ \underline{s}_i \in X \\  \underline{t}_j \in Y}} (-1)^{i+j}\frac{|G_{\underline{s}_i}\cap {G_{\underline{t}_j}}|}{|G|}\big[G_{\underline{s}_i}\cap {G_{\underline{t}_j}}, \Res^{G_{s_0}}_{G_{\underline{s}_i}}(\mathfrak{l}_{s_0}) \Res^{G_{t_0}}_{G_{\underline{t}_j}}(\mathfrak{m}_{t_0}) \big]_G.
 \end{align*}
 
 Thus, 	$\Lambda_{(X \times Y,  \mathfrak{l} \times \mathfrak{m})}= \Lambda_{(X, \mathfrak{l})}\Lambda_{(Y, \mathfrak{m})}.$	 
	
\end{proof}

The first assertion of Lemma \ref{lefprp} tells us that
every positive element of {\rouge $B_C(G)$ is} in of the form
$\Lambda_{(X,\mathfrak{l})}$ for some $C$-monomial
$G$-poset $(X,\mathfrak{l})$.
Now consider the poset 
$X= \{a,b,c,d,e\}$ with the ordering
 $\{a\leq c, a\leq d, a \leq e, b\leq c, b\leq d, b\leq e \}$.
  Consider trivial $G$-action on $X$.
Then $\Lambda_{\tau_G(X)}= -1_{B_C(G)}.$
So as a consequence of Lemma \ref{lefprp} we 
get the following corollary.

\begin{cor}\label{lefex}
Any element of the monomial Burnside ring can
 be expressed as the Lefschetz invariant
  of some (non unique) monomial $G$-poset. 
\end{cor}

\begin{prop}
	Let $H$ be a subgroup of $G$. Given a $C$-monomial
	$H$-poset $(X,\mathfrak{l})$, we have
	$$\Ind^G_H(\Lambda_{(X,\mathfrak{l})})= \Lambda_{\Ind^G_H(X,\mathfrak{l})}.$$ 
\end{prop}
\begin{proof}Since
	$$\Ind^G_H(\Lambda_{(X,\mathfrak{l})})= 
	 \sum\limits_{n \in \NN}(-1)^n 
	 \Ind^G_H\big(\Sd_n(X),\mathfrak{l}_n\big),$$
	 we need to show that there exists a $C$-monomial
	 $G$-set isomorphism between
	 $$\big(G\times_H\Sd_n(X), G\times_H\mathfrak{l}_n\big)$$
	 and 
	 $$\big(\Sd_n(G\times_HX), (G\times_H\mathfrak{l})_n\big)$$
	 for any $n\in \NN.$\par
	 We define 
	 $$(f_n,\id): \big(G\times_H\Sd_n(X), G\times_H\mathfrak{l}_n\big)
	 \rightarrow \big(\Sd_n(G\times_HX), (G\times_H\mathfrak{l})_n\big)$$
	 where 
	 $$f_n: G\times_H\Sd_n(X) \rightarrow \Sd_n(G\times_HX)$$ 
	 such that
	 $$f_n(u,_{_H}x_0<...<x_n)= 
	 \big((u,_{_H}x_0)<...<(u,_{_H}x_n)\big)$$
	 for any chain $(u,_{_H}x_0<...<x_n)$
     in $G\times_H\Sd_n(X).$\par	 
	 Let  $(u_0,_{_H}x_0)<...<(u_n,_{_H}x_n)$ be a chain
	 in $\Sd_n(G\times_HX)$. There exist some $h_i \in H$
	 such that $u_ih_i=u_{i+1}$ and {\rouge $h_{i}^{-1}x_i< x_{i+1}$}
	 for all $0\leq i \leq n-1$. Then
	 $$f_n\big(u_0,_{_H}
	 x_0<h_0x_1<...<h_0...h_{n-1}x_n\big)
	 = (u_0,_{_H}x_0)<...<(u_n,_{_H}x_n).$$ 
	 Obviously, $f_n$ is a map of $G$-sets and injective.\par
	 Now, we show that $G\times_H\mathfrak{l}_n=(G\times_H\mathfrak{l})_n\circ f_n.$
	 We consider an element $k\in G$, and chains  $(u,_{_H}x_0<...<x_n)$ in $G\times_H\Sd_n(X)$ such that 
	 $$k(u,_{_H}x_0<...<x_n)= (v,_{_H}y_0<...<y_n).$$
	 There exists some $h\in H$ such that 
	 $ku=vh$ and $hx_i=y_i$ for all $0\leq i \leq n$. 
	 Then
	 \begin{align*}
	 &(G\times_H\mathfrak{l})_n
	 \big(k, f_n(u,_{_H}x_0<...<x_n),f_n(v,_{_H}y_0<...<y_n)\big)\\
	 &= (G\times_H\mathfrak{l})_n \big(k, (u,_{_H}x_0)<...<(u,_{_H}x_n), (v,_{_H}y_0)<...<(v,_{_H}y_n)\big)\\
	 &=\mathfrak{l}_n(h,x_0<...<x_n,y_0<...<y_n) \\
	 &=(G\times_H\mathfrak{l}_n)\big(k,(u,_{_H}x_0<...<x_n),(v,_{_H}y_0<...<y_n)\big).\\
	 \end{align*} 

 \end{proof}

{\rouge Let $(X,\mathfrak{l})$ be a $G$-poset and let $x\in X$. Then the pairs 
	$(]x,\cdot[_X,{\dgreen \mathfrak{l}_{>x}})$ and $(]\cdot, x[_X, {\dgreen \mathfrak{l}^{<x}})$ are $C$-monomial
	$G_x$-posets where
$$]x,\cdot[_X= \{y \in X\mid x<y\},\,\,\,\,\,\,\,\,\,\,]\cdot, x[_X= \{y \in X\mid y<x\}$$
which are $G_x$-posets 
and ${\dgreen \mathfrak{l}_{>x}} : \widehat{]x,\cdot[_X} \rightarrow \bullet_C$ and ${\dgreen \mathfrak{l}^{<x}} : \widehat{]\cdot, x[_X} \rightarrow \bullet_C$ are the restrictions of the functor $\mathfrak{l}$.

}

\begin{lemma}\label{forind}
	Let $(X, \mathfrak{l})$ be a monomial $G$-poset. We have
	
	$$\Lambda_{(X, \mathfrak{l})} = 
	- \!\!\!\!\sum\limits_{\substack{x \in [G/X]}}\!\!\!\!\Ind^G_{G_x}
	 \big( [G_x, \mathfrak{l}_{x}]_{G_x}\cdot {\rouge\widetilde{\Lambda}_{]x,\cdot [_X}}\big).$$	
\end{lemma}

\begin{proof}
	\begin{align*}
	\Lambda_{(X, \mathfrak{l})} &=\!\!\!\!\!\!\!\!\sum\limits_{x_0 < ... < x_n \in_{_G} X}\!\!\!\!\!\!\!\!(-1)^n \big[{\rouge G_{x_0, ..., x_n}}, \Res^{G_{x_0}}_{G_{x_0, ..., x_n}}(\mathfrak{l}_{x_0})\big]_G \\
	&=\!\!\!\sum\limits_{x_0  \in_{_G} X}\sum\limits_{x_1 < ... < x_n \in_{_G} X : {\rouge x_0 < x_1} }\!\!\!\!\!\!\!\!
	(-1)^n \big[G_{x_0, ..., x_n}, \Res^{G_{x_0}}_{G_{x_0, ..., x_n}}(\mathfrak{l}_{x_0})\big]_G\\
	&= \!\!\!\sum\limits_{x_0  \in_{_G} X}\Ind^G_{G_{x_0}}\sum\limits_{x_1 < ... < x_n \in_{_{G_{x_0}}} {\rouge ]x_0, \cdot[_X}}\!\!\!\!\!\!\!\!(-1)^n \big[{\rouge G_{x_0, ..., x_n}}, \Res^{G_{x_0}}_{G_{x_0, ..., x_n}}(\mathfrak{l}_{x_0})\big]_{G_{x_0}}\\
	&= \!\!\sum\limits_{x_0  \in_{_G} X}\!\!\Ind^G_{G_{x_0}}\big[G_{x_{0}}, \mathfrak{l}_{x_0}\big]_{G_{x_0}}
	\sum\limits_{x_1 < ... < x_n \in_{_{G_{x_{0}}}{\rouge ]x_0, \cdot[_X}}}(-1)^n \big[{\rouge G_{x_0, ..., x_n}}, 1_{G_{x_0,..,x_n}}\big]_{G_{x_0}}\\
	&=- \!\!\sum\limits_{\substack{\dgreen x  \in_{_G} X}}\!\!\Ind^G_{G_x} \big( [G_x, \mathfrak{l}_{x}]_{G_x}\cdot {\rouge \widetilde{\Lambda}_{]x,\cdot [_X}}\big).
	\end{align*}
	
\end{proof}

\begin{rem}\label{opposite} 
We can define the {\em opposite} of a $C$-monomial $G$-poset $(X,\mathfrak{l})$ as follows. We consider the pair
$(X^{\op}, \mathfrak{l}^{\op})$ where
$X^{\op}$ is the opposite $G$-poset with the order $\leq^{\op}$ defined by 
$$\forall x,\,y \in X,\,g\in G,\,\, gx\leq^{\op} y \Leftrightarrow y\leq gx$$
and $\mathfrak{l}^{\op}: \widehat{X^{\op}}\rightarrow \bullet_C$ is defined
by 
$$\mathfrak{l}^{\op}(g,x,y)=\mathfrak{l}^{-1}(g^{-1}, y, x)$$
for any $x,\,y \in X^{\op}$ and $g\in G$ such that $gx\leq^{\op}y$.
Obviously, the pair $(X^{\op}, \mathfrak{l}^{\op})$ is a $C$-monomial $G$-poset.
Moreover the assignment $(X,\mathfrak{l}) \mapsto (X^{\op}, \mathfrak{l}^{\op})$ 
is a functor $_CMG\text{-\sf poset} \rightarrow {_CMG\text{-\sf poset}}$: if
$(f,\lambda): (X,\mathfrak{l}) \rightarrow (Y,\mathfrak{m})$ is a map
of $C$-monomial $G$-posets, then $f:X^{\op} \rightarrow Y^{\op}$ is a map
of $G$-posets and for any $gx\leq^{\op}\xp$, we get the commutative diagram
{\dgreen
$$\xymatrix@C=7ex@R=7ex{
\mathfrak{l}(x)\ar[r]^-{\lambda_x}\ar[d]_-{\mathfrak{l}^{\op}(g, x, \xp)}&\mathfrak{m} \circ f(x)\ar[d]^-{\mathfrak{m}^{\op}\left(\rule{0ex}{1.25ex}g, f(x), f(\xp)\right)}\\
\mathfrak{l}(\xp)\ar[r]_-{\lambda_{\xp}}&\mathfrak{m} \circ f(\xp).
}
$$
}
%
%
%

Observe that $(\mathfrak{l}^{\op})_x(g)=\mathfrak{l}^{-1}(g^{-1},x,x)=\mathfrak{l}(g,x,x)=\mathfrak{l}_x(g)$, for any $x\in X$ and $g\in G_x$. It follows that $\Lambda_{(X, \mathfrak{l})}= \Lambda_{(X^{\op}, \mathfrak{l}^{\op})}$.\par
\end{rem}

Let  $(f, \lambda) : (X, \mathfrak{l}) \rightarrow (Y, \mathfrak{m})$ be
a map of $C$-monomial $G$-posets. Given $y\in Y$, following \cite{BURN} we set
$$f^y= \{x \in X\mid f(x)\leq y\},\,\,\,\,\,\,\,\,\,\,f_y=\{x\in X\mid f(x)\geq y\}$$ 
which are both $G_y$-posets. We denote by $(f^y, {\dgreen \mathfrak{l}_{|f^y}})$ the $C$-monomial {\rouge $G_y$}-poset where {\rouge${\dgreen \mathfrak{l}_{|f^y}} : \widehat{f^y} \rightarrow \bullet_C$} is the restriction of the functor $\mathfrak{l}$. Similarly, we denote by  
$(f_y, {\dgreen \mathfrak{l}_{|f_y}})$ to be $C$-monomial {\rouge $G_y$-poset} where 
{\rouge ${\dgreen \mathfrak{l}_{|f_y}}: \widehat{f_y} \rightarrow \bullet_C$} is the
restriction of the functor $\mathfrak{l}$.

\begin{example} Let $(f,\lambda): (X, \mathfrak{l}) \rightarrow (Y, \mathfrak{m}) $
	 be a map of $C$-monomial $G$-posets. We define a $G$-poset $X* _{f,\lambda} Y$ with underlying $G$-set $X\sqcup Y$ as follows: for $z,z'\in X\sqcup Y$, we set
	
		\[
	z \leq z' \Leftrightarrow 
	\begin{cases}
	z,z' \in X   &\hbox{\rm and}\;\;\;  z \leq z' \in X \\
	z,z'\in Y  &\hbox{\rm and} \;\;\;   z \leq z' \in Y \\
	z \in X, z' \in Y &\hbox{\rm and}\;\;\;  f(z) \leq z' \in Y
	\end{cases}.
	\]
	
	We define the functor  
	$\mathfrak{l}*_{f,\lambda}\mathfrak{m} : \widehat{X \sqcup Y} \rightarrow \bullet_C$ by

	\[ (\mathfrak{l}*_{f,\lambda} \mathfrak{m})(g,z,z^\prime)=
	\begin{cases}
	\mathfrak{l}(g,z,z^\prime)   &\hbox{\rm if}\,\,  z,\,z^\prime \in X \\
	\mathfrak{m}(g,z,z^\prime)      &\hbox{\rm if}\,\,  z,\,z^\prime \in Y \\
	\mathfrak{m}(g,f(z),z^\prime)\lambda_z & \hbox{\rm if}\,\,  z \in X,\,z^\prime \in Y. \\
	\end{cases}.
	\]
	for any $z,\,z^\prime \in  X* _{f,\lambda} Y$ and $g\in G$ such that $gz\leq z^\prime$.\par
	Now let $z_1,\,z_2,\,z_3 \in X* _{f,\lambda} Y$ and $g,\,g^\prime \in G$ such that $gz_1\leq z_2$ and
	$g^\prime z_2\leq z_3$. We aim to show that
	$$ (\mathfrak{l}*_{f,\lambda} \mathfrak{m})(g^\prime g,z_1,z_3)
	= (\mathfrak{l}*_{f,\lambda} \mathfrak{m})(g^\prime,z_2,z_3)
	(\mathfrak{l}*_{f,\lambda} \mathfrak{m})(g,z_1,z_2).$$ 
	We have four cases to consider:
	\begin{itemize}
		\item $z_1,\,z_2,\,z_3 \in X$
		\item $z_1,\,z_2 \in X$ and $z_3 \in Y$
		\item $z_1\in X$ and $z_2,\,z_3 \in Y$
		\item $z_1,\,z_2,\,z_3 \in Y$.
	\end{itemize}
 In the first case we get
 $$ (\mathfrak{l}*_{f,\lambda} \mathfrak{m})(g^\prime g,z_1,z_3)
 = \mathfrak{l}(g^\prime g,z_1,z_3)= 
 \mathfrak{l}(g^\prime, z_2,z_3)\mathfrak{l}(g,z_1,z_2)$$
$$ = (\mathfrak{l}*_{f,\lambda} \mathfrak{m})(g^\prime, z_2,z_3)
(\mathfrak{l}*_{f,\lambda} \mathfrak{m})(g,z_1,z_2).$$
In the second case, using the naturality of $\lambda$ we get
$$(\mathfrak{l}*_{f,\lambda} \mathfrak{m})(g^\prime g,z_1,z_3)
= \mathfrak{m}\big(g^\prime g, f(z_1), z_3\big)\lambda_{z_1}
= \mathfrak{m}\big(g^\prime, f(z_2), z_3\big)\mathfrak{m}\big(g, f(z_1), f(z_2)\big)\lambda_{z_1}$$
$$=\mathfrak{l}(g,z_1,z_2){\rouge \mathfrak{m}\big(g^\prime,f(z_2),z_3\big)}\lambda_{z_2}
= (\mathfrak{l}*_{f,\lambda} \mathfrak{m})(g^\prime, z_2, z_3)
(\mathfrak{l}*_{f,\lambda} \mathfrak{m})(g, z_1, z_2).$$
In the third case, we get
$$(\mathfrak{l}*_{f,\lambda} \mathfrak{m})(g^\prime g,z_1,z_3)
= \mathfrak{m}\big(g^\prime g,f(z_1),z_3\big)\lambda_{z_1}=
\mathfrak{m}\big(g^\prime ,f(z_2),z_3\big)
\mathfrak{m}\big(g,f(z_1),f(z_2)\big)\lambda_{z_1}$$
$$
=(\mathfrak{l}*_{f,\lambda} \mathfrak{m})(g,z_1,z_2)
(\mathfrak{l}*_{f,\lambda} \mathfrak{m})(g^\prime,z_2,z_3).$$
In the fourth case 
$$(\mathfrak{l}*_{f,\lambda} \mathfrak{m})(g^\prime g,z_1,z_3)
= \mathfrak{m}\big(g^\prime g,z_1,z_3\big)=
\mathfrak{m}(g^\prime, z_2,z_3)\mathfrak{m}(g,z_1,z_2)$$
$$ = (\mathfrak{l}*_{f,\lambda} \mathfrak{m})(g^\prime, z_2,z_3)
(\mathfrak{l}*_{f,\lambda} \mathfrak{m})(g,z_1,z_2).$$
Let $z\in  X* _{f,\lambda} Y$ then obviously we have
$(\mathfrak{l}*_{f,\lambda} \mathfrak{m})(1,z,z)= 1.$ Thus, 
$( X* _{f,\lambda} Y, \mathfrak{l}*_{f,\lambda} \mathfrak{m})$
is a $C$-monomial $G$-poset.
\end{example}

\begin{lemma}\label{st} Let $(f,\lambda) : (X, \mathfrak{l}) \rightarrow (Y, \mathfrak{m})$
	 be a map of $C$-monomial $G$-posets. Then
	
$\Lambda_{(X* _{f,\lambda} Y, \mathfrak{l}*_{f,\lambda} \mathfrak{m} )}= \Lambda_{(Y, \mathfrak{m})}.$
		
\end{lemma}
\begin{proof}
\begin{enumerate}
\item  Let {\rouge $z \in Z=X* _{f,\lambda} Y$}. If $z\in X$ consider the map $g:\, ]z, \cdot[_Z \rightarrow  [f(z), \cdot[_Y$ defined by  
	\[ g(t)=
\begin{cases}
f(t)  &\text{if}\,\,  t \in X \\
t     &\text{if}\,\,  t \in Y \\
\end{cases}.
\]

Let $g^\prime :\, \big[f(z), \cdot\big[ \rightarrow \big]z, \cdot\big[$ defined
 by $g^\prime(s)=s.$ Then $g$ and $g^\prime$ are maps of 
 $G_z$-posets
such that $g \circ g^\prime = \text{Id}$ and {\rouge $\text{Id} \leq g^\prime \circ g.$} So if $z \in X$
  using [\cite{BURN}, Lemma 4.2.4 and Proposition 4.2.5], we get
   $\widetilde{\Lambda}_{]z,\cdot [}= \widetilde{\Lambda}_{[f(z),\cdot [}=0.$ Thus,
\begin{align*}
\Lambda_{(X* _{f,\lambda} Y, \mathfrak{l}*_{f,\lambda} \mathfrak{m})}&=
 -\!\!\!\!\!\!\!\sum\limits_{\substack{z \in [G\backslash X* _{f,\lambda} Y]}}\!\!\!\!\!\!\!
 \Ind^G_{G_z} ( [G_z, \mathfrak{l}_{z}]_{G_z}\cdot \widetilde{\Lambda}_{]z,\cdot [})\\
&=-\!\!\!\!\sum\limits_{\substack{y \in [G\backslash Y]}}\!\!\!\!\Ind^G_{G_y}
 ( [G_y, \mathfrak{l}_{y}]_{G_y}\cdot \widetilde{\Lambda}_{]y,\cdot [})=\Lambda_{(Y, \mathfrak{m})}.
\end{align*}
\end{enumerate}
\end{proof}

As a consequence, we give an analogue of Proposition 4.2.7. in \cite{BURN}, which in turn was inspired by a much deeper theorem of Quillen in~\cite{quillen}.

\begin{prop} Let $(f,\lambda) : (X, \mathfrak{l}) \rightarrow (Y, \mathfrak{m}) $ be a map
	 of $C$-monomial $G$-posets. Then in $B_C(G)$
	
	$$\widetilde{\Lambda}_{(Y,\mathfrak{m})}= \widetilde{\Lambda}_{(X,\mathfrak{l})} 
	+ \!\!\!\!\sum\limits_{y \in G\backslash Y}\!\!\!\!
	\Ind^G_{G_y}(\widetilde{\Lambda}_{f^y}
	{\rouge \widetilde{\Lambda}_{(]y,\cdot[_Y, {\dgreen\mathfrak{m}^{>y}})}}).$$
	
		$$\widetilde{\Lambda}_{(Y,\mathfrak{m})}= \widetilde{\Lambda}_{(X,\mathfrak{l})} 
	+ \!\!\!\!\sum\limits_{y \in G\backslash Y}\!\!\!\!
	\Ind^G_{G_y}(\widetilde{\Lambda}_{f_y}
	{\rouge\widetilde{\Lambda}_{(]\cdot,y[_Y, {\dgreen\mathfrak{m}_{<y}} )}}).$$

\end{prop}

\begin{proof}
	We follow the proof of Proposition 4.2.7 in \cite{BURN}.
	For any $n \in \NN$, any chain $z=z_0 < ... < z_n \in \Sd_n(X* _{f,\lambda} Y)$ can be of two
	types, depending on $z_n \in X$ or $z_n \in Y$. For a sequence $z$ of the first type  we get
	$$\Res^{G_{z_n}}_{G_{z_0,...,z_n}}( {\mathfrak{l}*_{f,\lambda} \mathfrak{m})_n}_{z_n}
	= \Res^{G_{z_n}}_{G_{z_0,...,z_n}}\mathfrak{l}_n.$$	
	Now a sequence $z$ of the second type has a smallest element $y=z_i$ in $Y$, thus, we can write the sequence as 
	$$x_0 < ... < x_{i-1} < y < y_0 < ... < y_{n-i-1}$$
	such that $x_0 < ... < x_{i-1}$ is in $\Sd_{i-1}(f^y)$, and $y_0 < ... < y_{n-i-1}$ 
	is in $\Sd_{n-i-1}(]y,\cdot[_Y)$.  We get
	$$\Res^{G_{z_n}}_{G_{z_0,...,z_n}}( {\mathfrak{l}*_{f,\lambda} \mathfrak{m})_n}_{z_n} 
	= \Res^{G_{y}}_{G_{z_0,...,z_n}}(\mathfrak{m}).$$
	Let $\underline{x}_{i-1}$ denote the chain $x_0 < ... < x_{i-1}$ and
	$\underline{y}_{n-i-1}$ denote the chain $y_0 < ... < y_{n-i-1}$.
     Then, by Lemma \ref{lefprp} and Lemma \ref{st} we get
     $${\rouge\Lambda_{(Y,  \mathfrak{m})}}= 
     \Lambda_{(X* _{f,\lambda} Y, \mathfrak{l}*_{f,\lambda} \mathfrak{m})}
     = \sum\limits_{n \in \NN}(-1)^n
     \big(\Sd_n(X* _{f,\lambda} Y), ( \mathfrak{l}*_{f,\lambda} \mathfrak{m})_n\big)$$
     $$=\!\!\!\!\!\!\!\!\!\!\!\!\sum\limits_{\substack{n \in \NN\\ z_0 < ... < z_n \in \Sd_n(\mathfrak{l}*_{f,\lambda} \mathfrak{m})}}
     \!\!\!\!\!\!\!\!\!\!\!\!(-1)^n\big[G_{z_0,...,z_n}, \Res^{G_{z_n}}_{G_{z_0,...,z_n}}( \mathfrak{l}*_{f,\lambda} \mathfrak{m})_n\big]_G$$	
      $$=\sum\limits_{n\in \NN}(-1)^n\big(\Sd_n(X), \mathfrak{l}_n\big)$$
        $$+ \!\! \sum\limits_{y \in [G\dom Y]} \!\!\!\!\!\Ind^G_{G_y}\sum\limits_{i=0}^n
      \!\!\!\!\! \sum\limits_{\substack{\underline{x}_{i-1}\in \Sd_{i-1}(f^y)\\
       		\underline{y}_{n-i-1}\in \Sd_{n-i-1}(]y,\cdot[_{G_y})}}
      \!\!\!\!\!\!\!\!\!\!\!\!\!\!\!\!\!\!
      \big[G_{\underline{x}_{i-1},y,\underline{y}_{n-i-1}}, 
      \Res^{G_{y}}_{G_{\underline{x}_{i-1},y,\underline{y}_{n-i-1}}}\mathfrak{m}_y\big]_{G_y}$$	
      $$=\Lambda_{(X, \mathfrak{l})}+ \sum\limits_{y \in [G\dom Y]} \!\!\!\!\!\Ind^G_{G_y}
      \big(\widetilde{\Lambda}_{(f^y,1_{f^y})}{\rouge \widetilde{\Lambda}_{(]y,\cdot[_Y, {\dgreen \mathfrak{m}_{>y}})}}\big).$$
     For the second assertion we consider the opposite map 
     $$(f,\lambda): (X^{\op},\mathfrak{l}^{\op}) \rightarrow (Y^{\op},\mathfrak{m}^{\op})$$
     Since we have $\Lambda_{(X, \mathfrak{l})}= \Lambda_{(X^{\op}, \mathfrak{l}^{\op})}$ by Remark~\ref{opposite},
     the result follows.
\end{proof}
\begin{cor} Let $(f,\lambda): (X, \mathfrak{l}) \rightarrow (Y, \mathfrak{m}) $ be a map of $C$-monomial $G$-posets. If $\Lambda_{f^y}=0$ {\rouge for all} $y\in Y$ (resp. if $\Lambda_{f_y}=0$ {\rouge for all} $y\in Y$), then $\Lambda_{X, \mathfrak{l}}=\Lambda_{Y, \mathfrak{m}}$. 
\end{cor}
\begin{rem}
The assumption of this corollary is fulfilled in particular if $f:\widehat{X}\to \widehat{Y}$ admits a right adjoint $g$, in other words if there exists a map of posets $g:Y\to X$ such that $f(x)\leq y \Leftrightarrow x\leq g(y)$ for any $x\in X$ and $y\in Y$, i.e. equivalently if $f\circ g(y)\leq y$ and $g\circ f(x)\leq x$ for any $x\in X$ and any $y\in Y$.
\end{rem}
Now we set some notation. Given a $C$-monomial $G$-set $(X,\mathfrak{l})$, we can rewrite
its Lefschetz invariant as 
$$\Lambda_{(X, \mathfrak{l})}=
\!\!\!\!\!\!\!\!\!\sum\limits_{x_0 < ... < x_n \in_G X}\!\!\!\!\!\!\!\!
(-1)^n \big[G_{x_0, ... x_n}, \Res^{G_{x_0}}_{G_{x_0, ..., x_n}}( \mathfrak{l}_{x_0})\big]_G$$
$$= \!\!\!\!\!\!\!\!\!\!\sum\limits_{(V,\nu)\in_G \ch(G)}\!\!\!\!\!\!\!\!
\gamma^{X,\mathfrak{l}}_{V,\nu}[V,\nu]_G$$

where 
$$\gamma^{X,\mathfrak{l}}_{V,\nu}= 
\!\!\!\!\!\!\!\!\!\!\!\sum\limits_{\substack{x_0 < ... < x_n \in_G X \\ 
		(G_{x_0, ..., x_n}, \Res^{G_{x_0}}_{G_{x_0, ..., x_n}}\mathfrak{l}_{x_0})
		=_G (V,\nu)}}\!\!\!\!\!\!\!\!\!\!\!\!\!\!\!\!\!\!\!\! (-1)^n
	= \frac{1}{\rouge|N_G(V,\nu):V|}
	\!\!\!\!\!\!\!\sum\limits_{\substack{x_0 < ... < x_n \in X \\
	(G_{x_0, ..., x_n}, \Res^{G_{x_0}}_{G_{x_0, ..., x_n}}\mathfrak{l}_{x_0})
	= (V,\nu)}}\!\!\!\!\!\!\!\!\!\!\!\!\!\!\!\!\!\!\!\!\!(-1)^n.$$

Given a $C$-monomial $G$-poset $(X, \mathfrak{l})$ 
we let the set $(X, \mathfrak{l})^{U,\mu}$ to be

$$(X, \mathfrak{l})^{U,\mu}=\{x \in X^U \mid \Res^{G_x}_{U}\mathfrak{l}_x=\mu \}$$

where $(U, \mu)$ is a subcharacter of $G$.
Then given a $C$-subcharacter $(U,\mu)\in\ch(G)$ we have
$$\chi\big((X, \mathfrak{l})^{U,\mu}\big)= 
\!\!\!\!\!\!\!\!\!\sum\limits_{\substack{n \in \NN \\ 
x_0 < ... < x_n \in X^U \\ \Res^{G_{x_0}}_{G_{x_0, ..., x_n}}\mathfrak{l}_{x_0}=\mu}}\!\!\!\!\!\!\!\!\!
(-1)^n
=\!\!\!\! \sum\limits_{\substack{(V,\nu) \in \ch(G)\\ U\subseteq V \\ \Res^V_U\nu=\mu}}\!\!\!\!
{\rouge m^{X,\mathfrak{l}}_{V,\nu}}$$

where 
$$m^{X,\mathfrak{l}}_{V,\nu}=\!\!\!\!\!\!\!\!\!\!\!\!\!\!\!\!\!\!\!\! \sum\limits_{\substack{n \in \NN \\ 
		x_0 < ... < x_n \in X \\ (G_{x_0, ..., x_n},\Res^{G_{x_0}}_{G_{x_0, ..., x_n}}
		\mathfrak{l}_{x_0})=(V,\nu)}}\!\!\!\!\!\!\!\!\!\!\!\!\!\!\!\!\!\!\!\!\!\!\!\!\!\!\!\!
	(-1)^n.$$
Now $|N_G(V,\nu):V|m^{X,\mathfrak{l}}_{V,\nu}= \gamma^{X,\mathfrak{l}}_{V,\nu}.$ Using this fact
we prove the following lemma.
\begin{lemma}\label{Lefeu} Let $(X, \mathfrak{l})$ and $(Y, \mathfrak{m})$ 
	be $C$-monomial $G$-posets then $\Lambda_{(X, \mathfrak{l})} =
	 \Lambda_{(Y, \mathfrak{m})}$ if and only if 
	 $\chi\big((X, \mathfrak{l})^{U, \mu}\big)= \chi\big((Y, \mathfrak{m})^{U, \mu}\big)$ for {\rouge every} $C$-subcharacter $(U, \mu)$ of $G$.
\end{lemma}
\begin{proof}
Assume $\Lambda_{(X, \mathfrak{l})}= \Lambda_{(Y, \mathfrak{m})}.$ Then
$$\sum\limits_{(V,\nu)\in_G \ch(G)}\!\!\!\!\!\!\!
\gamma^{X,\mathfrak{l}}_{V,\nu}[V,\nu]_G
= \!\!\!\!\!\!\!\!\sum\limits_{(V,\nu)\in_G \ch(G)}\!\!\!\!\!\!\!
\gamma^{Y,\mathfrak{m}}_{V,\nu}[V,\nu]_G$$
$$\sum\limits_{(V,\nu)\in_G \ch(G)}\!\!\!\!\!\!\!
(\gamma^{X,\mathfrak{l}}_{V,\nu}-\gamma^{Y,\mathfrak{m}}_{V,\nu})[V,\nu]_G=0.$$
So $\gamma^{X,\mathfrak{l}}_{V,\nu}=\gamma^{Y,\mathfrak{m}}_{V,\nu}$ and then 
$m^{X,\mathfrak{l}}_{V,nu}= m^{Y,\mathfrak{m}}_{V,nu}$ for {\rouge every} $C$-subcharacter $(V,\nu)$
of $G$. We get 
$$\sum\limits_{(U,\mu)\leq(V,\nu) \in_{G} \ch(G)}\!\!\!\!\!\!\!\!\!\!\!\!\!
m^{X,\mathfrak{l}}_{V,\nu}
= \!\!\!\sum\limits_{(U,\mu)\leq(V,\nu) \in_{G} \ch(G)}\!\!\!\!\!\!\!\!\!\!\!\!\!
m^{Y,\mathfrak{m}}_{V,\nu}.$$
Thus, $\chi\big((X, \mathfrak{l})^{U,\mu}\big)= \chi\big((Y, \mathfrak{m})^{U,\mu}\big)$
for {\rouge every} $C$-subcharacter $(U,\mu)$ of $G$.\par
Conversely, assume that $\chi\big((X, \mathfrak{l})^{U,\mu}\big)= \chi\big((Y, \mathfrak{m})^{U,\mu}\big)$
for {\rouge every} $C$-subcharacter $(U,\mu)$ of $G$. Then
$$\sum\limits_{(U,\mu)\leq(V,\nu) \in \ch(G)}\!\!\!\!\!\!\!\!\!\!\!\!\!
m^{X,\mathfrak{l}}_{V,\nu}
= \sum\limits_{(U,\mu)\leq(V,\nu) \in \ch(G)}\!\!\!\!\!\!\!\!\!\!\!\!\!
m^{Y,\mathfrak{m}}_{V,\nu},$$
$$\sum\limits_{(U,\mu)\leq(V,\nu) \in \ch(G)}\!\!\!\!\!\!\!\!\!\!\!\!\!
(m^{X,\mathfrak{l}}_{V,\nu}-{\rouge m^{Y,\mathfrak{m}}_{V,\nu}})=0.$$
Let $z$ be the matrix with the coefficients 
$$z(U,\mu;V,\nu)= \lfloor (U,\mu) \leq (V,\nu) \rfloor = 
\begin{cases}
1& \text{if}\,\, (U,\mu) \leq (V,\nu)\\
0& \text{otherwise}
\end{cases}$$ for any $C$-subcharacters $(U,\mu),\,(V,\nu)$. If we list the $C$-subcharacters
in non-decreasing order of size of the subgroups, the matrix $z$ is upper triangular with
nonzero diagonal coefficients. Thus, $z$ is nonsingular and so 
$m^{X,\mathfrak{l}}_{V,\nu}={\rouge m^{Y,\mathfrak{m}}_{V,\nu}.}$ This implies 
$\gamma^{X,\mathfrak{l}}_{V,\nu}={\rouge\gamma^{Y,\mathfrak{m}}_{V,\nu}}$. We get
$$\Lambda_{(X,  \mathfrak{l})}= \sum\limits_{(V,\nu)\in_G \ch(G)}\!\!\!\!\!\!\!
\gamma^{X,\mathfrak{l}}_{V,\nu}[V,\nu]_G
= \sum\limits_{(V,\nu)\in_G \ch(G)}\!\!\!\!\!\!\!
\gamma^{Y,\mathfrak{m}}_{V,\nu}[V,\nu]_G
=\Lambda_{(Y,  \mathfrak{m})}.$$ 
This proves the lemma.
\end{proof}

\section{Generalized tensor induction}

Let $G$ and $H$ be finite groups. A set $U$ is a $(G,H)$-biset
if $U$ is a left $G$-set and right $H$-set such that the $G$-action
and the $H$-action commute. Any $(G,H)$-biset $U$ is a left
$(G\times H)$-set with the following action:
$$\forall u \in U,\,(g,h) \in G\times H\,\,\, (g,h)\cdot u= guh^{-1}.$$
A  $C$-monomial $(G\times H)$-set $(U,\lambda)$ will be called
a $C$-monomial $(G,H)$-biset, and usually denoted by $U_\lambda$ for simplicity.\par
Now let $U_\lambda$ be  a $C$-monomial $(G\times H)$-set and $u,\,\up \in U$. Then the set of morphisms from $u$ to $\up$ in $\widehat{U}$ is
$$\Hom_{\widehat{U}}(u,\up)=\{(g,h)\in G\times H\mid gu = \up h\}.$$
If $(g,h) \in \Hom_{\widehat{U}}(u, \up)$, 
we denote the image of $(g,h)$
under $\lambda$ by $\lambda(g,h,u,\up)$.\par
Let $U_\lambda$ be a  $C$-monomial $(G,H)$-biset
and $V_\rho$  be a $C$-monomial $(H,K)$-biset. 
Consider the set
$$U_\lambda\circ V_\rho=\{(u,v)\in U\times V\mid \forall h \in H_u\cap H_v,\, \lambda(1,h,u,u)\rho(h,1,v,v)=1\}.$$
The set $U_\lambda\circ V_\rho$ is an $H$-set with the action 
$$\forall (u,v)\in U_\lambda\circ V_\rho,\, \forall h \in H,\, h(u,v)= (uh^{-1}, hv).$$
Indeed, the condition that we impose on $U_\lambda\circ V_\rho$ amounts to saying that given $(u, v) \in U_\lambda\circ V_\rho,$ the 
	linear character {\dgreen $\xi_{u,v}:h\mapsto \lambda(1,h,u,u)\rho(h,1,v,v)$ of $H_u\cap H_v$ is  trivial. Moreover we have $\xi_{ux,x^{-1}v}(h)=\xi_{u,v}(xhx^{-1})=1$ for $x\in H$ and $h\in H_{ux}\cap H_{x^{-1}v}$, i.e. $xhx^{-1}\in H_u\cap H_v$.}
	\par
We let $U_\lambda\circ_H V_\rho$ denote the set of $H$-orbits on $U_\lambda\circ V_\rho$ and $(u,_{_H}v)$ denote the
$H$-orbit containing $(u,v)$. The set $U_\lambda\circ_H V_\rho$
is $(G,K)$-biset with the action 
$$(u,_{_H}v)\in U_\lambda\circ_HV_\rho,\, (g,k)\in G\times K,\, g(u,_{_H}v)k = (gu,_{_H}vk).$$
We obtain a $C$-monomial $(G,K)$-biset $(U_\lambda\circ_H V_\rho, \lambda\times\rho)$, where $\lambda\times\rho$ is defined as follows: if $(u,_{_H}v)$ and $(u',_{_H}v')\in U_\lambda\circ_HV_\rho$ and $(g,k)\in G\times K$ are such that $g(u,_{_H}v)=(u',_{_H}v')k$, then there exists $h\in H$ such that $gu=u'h$ and $hv=v'k$. This element $h$ need not be unique, but it is well defined up to multiplication on the right by an element of $H_u\cap{H_v}$. We set
$$(\lambda\times\rho)\big(g,k,(u,_{_H}v),(u',_{_H}v')=\lambda(g,h,u,u')\rho(h,k,v,v'),$$
which does not depend on the choice of $h$, by the defining property of $U_\lambda\circ V_\rho$.
Note that {\dgreen $U_\lambda\circ_H V_\rho = U\times_HV$ when $V$ is a left free $(H,K)$-biset, or when $\lambda$ and $\rho$ are both equal to the trivial functor}.\medskip\par
Given a $C$-monomial $G$-poset $(X, \mathfrak{l})$, we let $t_{U, \lambda}(X, \mathfrak{l})$ be the set of $G$-equivariant maps $f: U \rightarrow X $ such that 
$$\mathfrak{l}\big(g, f(u), f(u)\big)= \lambda(g,1,u,u)$$
 for all  $ u \in U$ and 
{\rouge $g \in G_u$.} Then $t_{U, \lambda}(X, \mathfrak{l})$ is an $H$-poset with the action $(hf)(u) = f(uh)$, 
for any $h \in H,$ for any $f \in t_{U, \lambda}(X, \mathfrak{l}),$ for any $u \in U$. The order $\leq$ is given as follows:

$$ \forall f, \fp \in  t_{U, \lambda}(X, \mathfrak{l}),\, f \leq \fp \Leftrightarrow \forall u \in U, f(u)\leq f'(u) \, \text{in}\, X.$$

Now we define a functor $\mathfrak{L}_{U, \lambda} : \widehat{t_{U, \lambda}(X, \mathfrak{l})} \rightarrow \bullet_C$. Let $f$, $ \fp \in t_{U, \lambda}(X, \mathfrak{l})$ and $h \in H$ 
such that $hf \leq \fp$. We choose a set $[G\backslash U]$ of representatives of $G$-orbits of $U$. Then for all $u \in U$ there exist some
 $\ghu \in G$  and a unique $\sshu \in [G\backslash U]$ such that 
 $$uh= \ghu \sshu.$$ 
Since $hf \leq f'$, we get $\ghu f\big(\sshu\big)\leq \fp(u),$  and we set

$$\mathfrak{L}_{U, \lambda}(h, f, \fp)= \prod\limits_{u \in [G\backslash U]} \mathfrak{l}\Big(\ghu, f\big(\sshu\big), \fp\big(u\big)\Big)\lambda^{-1}\big(\ghu,h,\sshu,u\big).$$

Now we show that this definition  does not depend on the choice of $\ghu$. Assume that there exist 
$\ghu$, $\gphu \in G$ such that 
$$uh= \ghu\sshu=\gphu \sshu.$$ 
So there exists $w \in G_{\sshu}$ such that $\ghu = \gphu w$. We get
$${\rouge \mathfrak{l}\Big(w,f\big(\sshu\big),f\big(\sshu\big)\Big)= \lambda\big(w,1,\sshu,\sshu\big).}$$

Furthermore, we get the following commutative diagram: 
{\dgreen
$$\xymatrix@R=4ex@C=3ex{
\sigma_h(u)\ar[rr]^-w\ar[rd]_-{g_{h,u}}&&\sigma_h(u)\ar[ld]^{g'_{h,u}}\\
&u&
}
$$
}
%
%


Thus,
\begin{align*}
\mathfrak{L}_{U, \lambda}(h, f, \fp)&= \prod\limits_{u \in [G\backslash U]} \mathfrak{l}\Big(\ghu, f\big(\sshu\big), \fp\big(u\big)\Big)\lambda^{-1}\big(\ghu, h,\sshu,u\big) \\
&=\prod\limits_{u \in [G\backslash U]} \mathfrak{l}\Big(\gphu w, f\big(\sshu\big), \fp\big(u\big)\Big)\lambda^{-1}\big(\gphu w, h,\sshu,u\big) \\
&=\prod\limits_{u \in [G\backslash U]} \mathfrak{l}\Big(\gphu, f\big(\sshu\big), \fp\big(u\big)\Big)\lambda^{-1}\big(\gphu , h,\sshu,u\big).\\
\end{align*}
\begin{defn}
The above construction $T_{U,\lambda}: (X,\mathfrak{l}) \mapsto \big(t_{U, \lambda}(X, \mathfrak{l}), \mathfrak{L}_{U, \lambda}\big)$ is called the
 {\em generalized tensor induction for $C$-monomial $G$-posets}, associated to $(U,\lambda)$. 
\end{defn}
\begin{lemma}
Let $G$ and $K$ be finite groups and $U$ be a $(G,K)$-biset. Then there exists a bijection
between the sets 
$\{(u,t)\mid u \in [G\backslash U /K],\, t \in  [(K\cap G^u) \backslash K]\}$ and $[G\backslash U]$.
\end{lemma}
\begin{proof}
Let $u \in [G\backslash U /K]$ and $t \in [(K\cap G^u) \backslash K]$ then there exist some
$\gtu \in G$ and a unique $\sstu \in [G\backslash U]$ such that 
$$ut= \gtu\sstu.$$ 
We define
$\psi : \{(u,t)\mid u \in [G\backslash U /K],\, t \in  [(K\cap G^u) \backslash K]\} \rightarrow [G\backslash U]$
by $\psi(u,t)=\sstu$.

\end{proof}
\begin{lemma}Let $G$ and $H$ be finite groups, $(U,\lambda)$ be a monomial $(G, H)$-biset and $(X,\mathfrak{l})$ be a $C$-monomial $G$-poset. 
	\begin{enumerate}
		    \item $\big(t_{U, \lambda}(X, \mathfrak{l}), \mathfrak{L}_{U, \lambda}\big)$ is a $C$-monomial $H$-poset.
			\item $\big(t_{U, \lambda}(X, \mathfrak{l}), \mathfrak{L}_{U, \lambda}\big)$ does not depend on the choice of representative set $[G\backslash U]$, up to isomorphism.
    \end{enumerate}
\end{lemma}

\begin{proof}
 \begin{enumerate}
 	\item We show that $\mathfrak{L}_{U, \lambda} : \widehat{t_{U, \lambda}(X, \mathfrak{l})} \rightarrow \bullet_{C}$
 	is a functor. Let $h$, $\hp \in H$ and {\rouge $f,\,\fp,\,\fpp \in t_{U, \lambda}(X,\mathfrak{l})$}
 	such that $hf \leq \fp$ and $\hp \fp\leq \fpp$. Let $u \in [G\backslash U]$. Then there exist   some 
 	$\ghu,\, \ghpu,\,\ghphu$ in $G$ and unique elements $\sshu,\, \sshpu,\, \sshphu$ in $[G\backslash U]$ such that
 	 $$uh= \ghu \sshu,\, u\hp = \ghpu \sshpu, \, u\hp h= \ghphu \sshphu.$$
 	  Also there exist some 
 	  $\ghsshpu \in G$ and a unique $\sshsshpu \in [G\backslash U]$ such that 
 	  $$\sshpu h = \ghsshpu \sshsshpu.$$ 
 	  Now we get 
 	  $$u\hp h = \ghpu \ghsshpu \sshsshpu$$ 
 	  and 
 	  $$\sshphu= \sshsshpu.$$ 
 	  Then there exists $w \in G_{\sshphu}$ such that 
 	  $${\rouge \ghphu = \ghpu\ghsshpu  w.}$$ 
 	  
 	  We have the following commutative diagram:
{\dgreen
$$\xymatrix@R=6ex@C=3ex{ 	  
\sigma_{h'h}(u)\ar[rr]^-w\ar[rd]_-{g_{h'h,u}}&&\sigma_{h'h}(u)\ar[ld]^{g_{h',u}g_{h,\sigma_{h'}(u)}}\\
&uh'h&
}
$$
}
%
%
%
%
%
 	On the other hand since  $w \in G_{\sshphu}$, we get
 	$$\mathfrak{l}\Big(w, f\big(\sshphu\big), f\big(\sshphu\big)\Big)=\lambda\big(w,1,\sshphu,\sshphu\big).$$
 	Thus, setting $L=\mathfrak{L}_{U, \lambda}(\hp h,f,\fpp)$, we have 
 	\begin{align*}
   L&= \prod\limits_{u \in [G \backslash U]}\mathfrak{l}{\rouge\Big(\ghphu,f\big(\sshphu\big) , \fpp \big(u\big)\Big)\lambda^{-1}\big(\ghphu,\hp h, \sshphu, u\big)}\\
   &=  \prod\limits_{u \in [G \backslash U]}{\rouge \mathfrak{l}\Big(\ghpu\ghsshpu  w,f\big(\sshphu\big) , \fpp \big(u\big)\Big)\lambda^{-1}\big(\ghpu\ghsshpu  w, \hp h, \sshphu, u\big)}\\
     &= \prod\limits_{u \in [G \backslash U]}\mathfrak{l}\Big(\ghpu \ghsshpu ,f \big(\sshphu\big), \fpp\big(u\big)\Big) \lambda^{-1}\big(\ghpu\ghsshpu, \hp h, \sshphu,u\big)\\
     &= \mathfrak{L}(\hp,\fp,\fpp)\mathfrak{L}(h,f,\fp).\\
    \end{align*}

      Moreover, given $f \in T_{U, \lambda}(X, \mathfrak{l})$ we have 
      $$\mathfrak{L} (1,f,f)= \prod\limits_{u \in [G \backslash U]}\mathfrak{l}\big(1, f(u), f(u)\big)\lambda^{-1}(1,1,u,u)= 1.$$
      
       Thus, $\mathfrak{L}_{U, \lambda} : \widehat{t_{U, \lambda}(X, \mathfrak{l})} \rightarrow \bullet_{C}$ is a functor.
 	
 	\item Let $h \in H$ and $f$, $\fp \in t_{U, \lambda}(X, \mathfrak{l})$ such that $hf \leq \fp$.
 	 Let $S=[G\backslash U]$ and let $S^\prime$ be the another choice of representatives. If $u^\prime \in S^\prime$ then 
 	there exist some $a_u \in G$, and a unique $u \in S$ such that $u^\prime= a_uu$. Then there exist some $g_{h,a_uu}$, 
 	$\ghu \in G$, a unique ${\rouge\sigma_{h}^{\prime}(a_uu)} \in S^\prime$, and a unique $\sshu \in S$ such that 
 	$$a_uuh= g_{h, a_uu}\sigma_{h}^{\prime}(a_uu)$$
 	 and $$uh= \ghu\sshu.$$
 	 Then
 	$$a_uuh= a_u\ghu\sshu= a_u\ghu a_{\sshu}^{-1}a_{\sshu}\sshu.$$
 	So $\sigma_{h}^{\prime}(a_uu)= a_{\sshu}\sshu$. Note that $a_{\sshu}\sshu \in S^\prime.$ 
 	We get the following commutative diagram:
{\dgreen
$$\xymatrix@R=8ex@C=10ex{
a_{\sshu}f\big(\sshu\big)\ar[r]^-{a_u\ghu a_{\sshu}^{-1}}\ar[d]_-{a_{\sshu}^{-1}}&a_u\fp\big(u\big)\\
f\big(\sshu\big)\ar[r]_-{g_{h,u}}&f(u).\ar[u]_-{a_u}
}
$$
}
%
%
%
 	
 	Thus, setting $L=\mathfrak{L}_{U, \lambda}^\prime (h,f,\fp)$, we have
 	
 	\begin{align*}
    L\!\! &=\! \! \prod\limits_{a_uu \in S^\prime}\!\! \mathfrak{l}\Big(a_u\ghu a_{\sshu}^{-1}, f\big(a_{\sshu}\sshu\big), \fp\big(a_uu\big)\Big)\lambda^{-1}\big(a_u\ghu a_{\sshu}^{-1}, h,a_{\sshu}\sshu, a_uu\big)\\
   &= \mathfrak{L}_{U, \lambda}^\prime(h,f,\fp)= \mathfrak{L}_{U, \lambda}(h,f,\fp)\alpha_{\fp}\alpha_{f}^{-1}\\
 		\end{align*}

 	 where $$\alpha_{\fp}= \prod\limits_{u \in S}\mathfrak{l}\big(a_u,\fp(u),a_u\fp(u)\big)\lambda^{-1}\big(a_u,1,u, a_uu\big)$$
 	  and 
 	  $$\alpha_{f}^{-1}= \prod\limits_{u \in S}\mathfrak{l}\big(a_{u}^{-1},a_{u}f(u),f(u)\big)\lambda^{-1}\big(a_u^{-1}, 1,a_uu, u\big).$$
 \end{enumerate}
\end{proof}

\begin{prop} Let $G$ and $H$ be finite groups and $(U,\lambda)$ be a $C$-monomial $(G, H)$-biset.
	\begin{enumerate}
		\item Let   $(X, \mathfrak{l}),\,(X^\prime, \mathfrak{l}^\prime)$ 
		be $C$-monomial $G$-posets then
		$$T_{U, \lambda}\big((X, \mathfrak{l})\times (X^\prime, \mathfrak{l}^\prime)\big) \cong 
		T_{U, \lambda}(X, \mathfrak{l})\times T_{U, \lambda}(X^\prime, \mathfrak{l}^\prime).$$
		\item $T_{U, \lambda} : {_CMG}${\sf-poset} $\rightarrow {_CMH}${\sf-poset} is a functor.
	\end{enumerate}

\end{prop}

\begin{proof}
	\begin{enumerate}
		\item is clear.
		\item Let $(\varphi, \beta): (X, \mathfrak{l}) \rightarrow (Y, \mathfrak{m})$
		be a map of $C$-monomial $G$-posets. We define 
		a map of $C$-monomial $G$-posets
		$$\big(T_{U, \lambda}(\varphi), T_{U, \lambda}(\beta)\big): \big(t_{U, \lambda}(X, \mathfrak{l}), \mathfrak{L}\big) \rightarrow \big(t_{U, \lambda}(Y, \mathfrak{m}), \mathfrak{M}\big)$$ 
		where 
		$$T_{U, \lambda}(\varphi) : t_{U, \lambda}(X, \mathfrak{l}) \rightarrow t_{U, \lambda}(Y, \mathfrak{m})$$
		such that $T_{U, \lambda}(\varphi)(f)= \varphi \circ f$ 
		and 
		$$T_{U, \lambda}(\beta) : \mathfrak{L}(f) \rightarrow \mathfrak{M} \circ T_{U, \lambda}(\varphi)(f)$$
		
		such that 
		$$T_{U, \lambda}(\beta)= \prod\limits_{u \in [G\backslash U]}\beta_{f(u)}$$
		for any $f \in  t_{U, \lambda}(X, \mathfrak{l})$. Clearly, 
		$\varphi \circ f : U \rightarrow X \rightarrow Y$ is 
		a map of $G$-posets. Since 
		given $g \in G_u$ and $u \in U$
		the map $\beta: \mathfrak{l} \rightarrow \mathfrak{m} \circ \varphi$ is natural,   we have the following commutative diagram:
		
{\dgreen
$$\xymatrix@R=8ex@C=8ex{
\mathfrak{l}\big(f(u)\big)\ar[r]^-{\beta_{f(u)}}\ar[d]_-{\mathfrak{l}\left(\rule{0ex}{1.25ex}g, f(u), f(u)\right)}&\mathfrak{m}\circ\varphi\big(f(u)\big)\ar[d]^-{\mathfrak{m}\left(\rule{0ex}{1.25ex}g, \varphi\circ f(u), \varphi\circ f(u)\right)}\\
\mathfrak{l}\big(f(u)\big)\ar[r]_-{\beta_{f(u)}}&\mathfrak{m}\circ\varphi\big(f(u)\big).
}
$$
}
%
%
%
%
		
	\end{enumerate}

So 
$$\beta_{f(u)}\mathfrak{l}\Big(g, f\big(u\big), f\big(u\big)\Big)= \mathfrak{m}\Big(g, \varphi\big(f(u)\big),\varphi\big(f(u)\big)\Big) \beta_{f(u)}.$$

Since $g \in G_{f(u)}$, we have
$$\mathfrak{l}\Big(g, f\big(u\big), f\big(u\big)\Big)= \lambda(g,1,u,u).$$
 Then we get 
$$\mathfrak{m}\Big(g, \varphi\big(f(u)\big), \varphi\big(f(u)\big)\Big)= \lambda(g,1,u,u).$$
Thus, {\rouge $\varphi \circ f \in t_{U, \lambda}(Y, \mathfrak{m})$.}

Now we show that 
$$T_{U, \lambda}(\beta): \mathfrak{L} \rightarrow \mathfrak{M}\circ T_{U, \lambda}(\varphi)$$

is a natural transformation. Let $f,\, \fp \in t_{U, \lambda}(X,\mathfrak{l})$ and $h \in H$
such that $hf \leq \fp$. We show that the following diagram is commutative:
{\dgreen
$$\xymatrix@R=6ex@C=10ex{
\mathfrak{L}\big(f\big)\ar[r]^-{T_{U, \lambda}(\beta)_{f}}\ar[d]_-{\mathfrak{L}(h, f, \fp)}&\mathfrak{M} \circ T_{U, \lambda}(\varphi)(f)\ar[d]^-{\mathfrak{M}(h, \varphi\circ f, \varphi\circ\fp)}\\
\mathfrak{L}\big(\fp\big)\ar[r]_-{T_{U, \lambda}(\beta)_{\fp}}&\mathfrak{M} \circ T_{U, \lambda}(\varphi)(\fp).
}
$$
}
%
%
%
%
%

Let {\rouge $u \in [G\backslash U]$. Then} there exist some $\ghu \in G$ and 
a unique $\sshu \in [G\backslash U]$
such that 
$$uh= \ghu\sshu.$$ 
Since $\beta: \mathfrak{l} \rightarrow \mathfrak{m} \circ \varphi$ is 
a natural transformation, we obtain the following commutative diagram :
{\dgreen
$$\xymatrix@R=8ex@C=10ex{
\mathfrak{l}\Big(f\big(\sshu\big)\Big)\ar[r]^-{\beta_{f(\sshu)}}\ar[d]_-{\rouge \mathfrak{l}\big(\ghu,\, f(\sshu),\, \fp(u)\big)}&\mathfrak{m} \circ \varphi\Big(f\big(\sshu\big)\Big)\ar[d]^-{\rouge \mathfrak{m}\Big(\ghu,\, \varphi\big(f\big(\sshu\big)\big),\, \varphi\big(\fp\big(u\big)\big)\Big)}\\
\mathfrak{l}\big(\fp(u)\big)\ar[r]_-{\beta_{\fp(u)}}&{\rouge \mathfrak{m} \circ \varphi\big(\fp(u)\big)}.
}
$$
}
%
%
%

Using the commutativity of the above diagram, and setting $T= T_{U, \lambda}(\beta)_{\fp}\circ\mathfrak{L}(h,f,\fp)$, we get

\begin{align*}
T&= T_{U, \lambda}(\beta)_{\fp}{\rouge\bigg(\prod\limits_{u \in [G \backslash U]} \mathfrak{l}\Big(\ghu, f\big(\sshu \big),\fp \big(u\big)\Big)\lambda^{-1}\big(\ghu, h, \sshu, u\big)\bigg)} \\
&= \prod\limits_{u \in [G \backslash U]} \beta_{\fp(u)} \mathfrak{l}\Big(\ghu, f\big(\sshu \big),\fp (u)\Big)\lambda^{-1}\big(\ghu, h, \sshu, u\big)\\
&= \prod\limits_{u \in [G \backslash U]} \mathfrak{m}\bigg(\ghu, \varphi\Big(f\big(\sshu\big)\Big), \varphi\Big(\fp\big(u\big)\Big)\bigg)\lambda^{-1}\big(\ghu, h, \sshu, u\big)\beta_{f\big(\sshu\big)}\\
&= \mathfrak{M}(h, \varphi \circ f, \varphi \circ \fp)\beta_f.\\
\end{align*}
So $T_{U, \lambda}(\beta): \mathfrak{L} \rightarrow \mathfrak{M}\circ T_{U, \lambda}(\varphi)$ is a natural transformation. Thus, 
$$\big(T_{U, \lambda}(\varphi), T_{U, \lambda}(\beta)\big): \big(t_{U, \lambda}(X, \mathfrak{l}), \mathfrak{L}\big) \rightarrow \big(t_{U, \lambda}(Y, \mathfrak{m}), \mathfrak{M}\big)$$

is a map of $C$-monomial $G$-posets.
\end{proof}
\begin{lemma}\label{bij} Let $G,\,H$ and $K$ be finite groups. If $U$ is a $(G, H)$-biset and
	$V$ is a left free $(H, K)$-biset{\dgreen, then the map $(u,v)\in U\times V\mapsto (u,_{_H}v)\in U\times_HV$ restricts to a bijection $\pi:[G\dom U]\times[H\dom V]\to [G\dom(U\times_HV)]$, where brackets denote sets of representatives of orbits.}
\end{lemma}
\begin{proof}
{\dgreen For $(u,v)\in U\times V$, there exists $v_0\in [G\dom V]$ and $h\in H$ such that $v=hv_0$. Then there exists $u_0\in[G\dom U]$ and $g\in G$ such that $uh=gu_0$. Then $(u,_{_H}v)=g(u_0,_{_H}v_0)$. Hence $\pi$ is surjective. Now if $(u_0,v_0)$ and $(u_1,v_1)$ are pairs in $[G\dom U]\times [H\dom V]$ which lie in the same $G$-orbit, there exists $g\in G$ and $h\in H$ such that $(gu_0,v_0)=(u_1h^{-1},hv_1)$. Hence $hv_1=v_0$, so $v_0=v_1=hv_1$, and $h=1$ since~$H$ act freely on $V$. Then $gu_0=u_1$, so $u_0=u_1$, and $\pi$ is injective. 
}
\end{proof}
\begin{prop}\label{PM}Let $G,\,H$ and $K$ be finite groups.
	\begin{enumerate}
		\item Let $(\bullet, 1)$ be the $C$-monomial $G$-poset where $\bullet$
		is $G$-poset with one element and $1: \bullet \rightarrow \bullet_C$ is the functor
		such that $1(g,\bullet, \bullet)=1$. Then $T_{U, \lambda}(\bullet, 1)= (\bullet, 1)$. 
\item Let $(\emptyset,z)$ be the empty $C$-monomial $(G,H)$-poset. Then  $T_{\emptyset,z}$ is the constant functor with value $(\bullet,1)$. 
		\item Let $(U, \lambda)$ and $(U^\prime, \lambda^\prime)$ be $C$-monomial $(G, H)$-bisets 
		and let  $(X, \mathfrak{l})$ be
		a $C$-monomial $G$-poset then  
		$$T_{U\sqcup U^\prime, \lambda\sqcup\lambda^\prime }(X, \mathfrak{l})=T_{U, \lambda}(X, \mathfrak{l})T_{U^\prime, \lambda^\prime}(X, \mathfrak{l}).$$
		\item Let $\id_G$ stand for the identity $(G, G)$-biset. Then 
		${\dgreen T _{\id_G,1}(X,\mathfrak{l})}=(X,\mathfrak{l})$ for any $C$-monomial $G$-poset $(X,\mathfrak{l}).$
		\item Let $(V, \rho)$ be a $C$-monomial left free $(H, K)$-biset, and $(U,\lambda)$ be a $C$-monomial $(H,G)$-biset. Then
		$$T_{V,\rho}\circ T_{U, \lambda} = T_{U \times_H V, \lambda\times\rho}.$$
	\end{enumerate}

\end{prop}
\begin{proof}

 	$1.$, $2.$, $3.$ and $4.$  are clear.
 
    $5.$ Note that since $V$ is left free, we have  $U_\lambda\circ_H V_\rho \cong {_GU\times_HV_K}$.
    Let $(X, \mathfrak{l})$ be a $C$-monomial $G$-poset.  We need to show that
    $$\Big(t_{V,\rho}\big(t_{U,\lambda}(X,\mathfrak{l}),\mathfrak{L}_{U,\lambda}\big), \mathfrak{L}_{V,\rho}\circ\mathfrak{L}_{U,\lambda}\Big)
    =\big(t_{U\times_HV,\lambda\times\rho}(X,\mathfrak{l}),
    \mathfrak{L}_{U\times_HV, \lambda\times\rho}\big).$$
    We define a $K$-poset map
    	$\varphi : t_{V,\rho}\big(t_{U,\lambda}(X,\mathfrak{l}),\mathfrak{L}_{U,\lambda}\big) \rightarrow t_{U \times_H V, \lambda\times\rho}(X, \mathfrak{l})$
    such that
    $$\varphi(f)(u,\sh v)= f(v)(u)$$
    for any $f \in t_{V,\rho}\big(t_{U,\lambda}(X,\mathfrak{l}),\mathfrak{L}_{U,\lambda}\big)$
    and $(u,\sh v)\in U \times_H V.$ It's clear that the map $\varphi(f)$ is a map of $G$-posets.

    Let $g \in G_{(u,_{_H} v)}$. Note that since $V$ is $H$-free,  we have $g \in G_u.$
 	Then 
 	$$\mathfrak{l}\Big(g, \varphi\big(f\big)(u,_{_H}v), \varphi\big(f\big)(u,_{_H}v)\Big)= \mathfrak{l}\Big(g, f(v)(u), f(v)(u)\Big)$$
   $$= \lambda(g,1,u,u)\rho(1,1,v,v)=(\lambda\times\rho)\big(g,1,(u,\sh v),(u,\sh v)\big).$$
 	and so 
 	$\varphi(f) \in t_{U \times_H V,\lambda\times\rho}(X, \mathfrak{l}).$
 	
 	Now we define a map 
 	$$\theta : t_{U \times_H V, \lambda\times\rho}(X, \mathfrak{l}) \rightarrow  t_{V,\rho}\big(t_{U,\lambda}(X,\mathfrak{l}),\mathfrak{L}_{U,\lambda}\big)$$ 
such that $\theta(t)(v)(u)= t(u,\sh v)$ for any $t \in  t_{U \times_H V, \lambda\times\rho}(X, \mathfrak{l})$,
 	$u \in U$ and $v\in V$. {\dgreen We show that}  $\theta(t)\in  t_{V,\rho}\big(t_{U,\lambda}(X,\mathfrak{l}),\mathfrak{L}_{U,\lambda}\big)$. {\dgreen Indeed, the map $\theta(t)$ is clearly a map of $H$-sets and moreover,} since $V$ is $H$-free, we have $H_v=1$ for any $v \in V$. Then
 	$$\mathfrak{L}_{U,\lambda}\big(1,\theta(t)(v),\theta(t)(v)\big)
 	=1=\rho(1,1,v,v).$$
 	
 	Clearly, $\theta(t)(v)$ is a map of {\dgreen $G$-sets}. Let $g \in G_u${\dgreen. Then $g \in G_{(u,\sh v)}$, and we} get
 	
    $$\mathfrak{l}\Big(g, \theta\big(t\big)(v)(u), \theta\big(t\big)(v)(u)\Big)= \mathfrak{l}\Big(g, t(u,\sh v), t(u,\sh v)\Big)$$
    	$$=\lambda(g,1,u,u)\rho(1,1,v,v)
    	=\lambda(g,1,u,u).$$
 	
 	 So $\theta(t) \in t_{V,\rho}\big(t_{U,\lambda}(X,\mathfrak{l}),\mathfrak{L}_{U,\lambda}\big)$. 
 	
 	Now we show that $\mathfrak{L}_{V,\rho}\circ\mathfrak{L}_{U, \lambda}={\dgreen \mathfrak{L}_{U \times_H V,\lambda\times\rho}}$.
 	Let $k \in K$ and  $f, \fp \in t_{V,\rho}\big(t_{U,\lambda}(X,\mathfrak{l}),\mathfrak{L}_{U,\lambda}\big)$ such that $kf \leq \fp$.
 	Let $v \in [H\backslash V]${\dgreen. Then} there exist a unique
 	$\sskv\in [H\backslash V]$ and some $\hkv \in H$ such that 
 	$${\rouge vk=\hkv\sskv.}$$ 
 	Let $u \in [G\backslash U]${\dgreen . Then} 
 	there exist a unique $\sshkvu \in [G\backslash U]$ and some $\ghkvu \in G$ such that
 	$$u\hkv= \ghkvu \sshkvu.$$ 
 	Then
 	$$(u,\sh v)= {\rouge (u\hkv\hkv^{-1},_{_H} v)= (u\hkv,_{_H}\hkv^{-1}v)}$$
 	$$= {\rouge \big(\ghkvu\sshkvu,_{_H}\sskv k^{-1}\big)}=\ghkvu\big(\sshkvu,_{_H}\sskv\big)k^{-1}.$$
We get
 	$$(u,\sh v)k= \ghkvu\big(\sshkvu,_{_H}\sskv\big).$$
 	Then 
 	$$\big(\sshkvu,_{_H}\sskv\big)= \sigma_k(u,\sh v)$$
 	and 
 	$$\ghkvu=g_{k,(u,\sh v)}w$$ for some $w \in G_{\sigma_k(u,\sh v)}$. We get the following commutative diagram:
{\dgreen
$$\xymatrix@R=6ex@C=3ex{
\sigma_k(u,\sh v)\ar[rr]^-w\ar[rd]_-{\ghkvu}&&\sigma_k(u,\sh v)\ar[dl]^-{g_{k,(u,\sh v)}}\\
&\rouge (u,\sh v)k 
}
$$
} 	
%
%
%
 	
Using the commutativity of the above diagram and Lemma \ref{bij} we get 
 	$$
 	\mathfrak{L}_{V, \rho} \circ \mathfrak{L}_{U, \lambda}(k, f, \fp)= 
 	\prod\limits_{v \in [H\backslash V]}\mathfrak{L}_{U, \lambda}\Big(\hkv,f\big(\sskv\big), \fp\big(v\big)\Big)\rho^{-1}\big(\hkv,k, \sskv, v\big)$$
 	
 	$$=\!\! \prod\limits_{\dgreen\substack{u \in [G\backslash U]\\v \in [H\backslash V]}}\!\!\mathfrak{l}\Big(\ghkvu,f\big(\sskv\big)(\sshkvu),\fp(v)(u)\Big)\lambda^{-1}\big(\ghkvu,\hkv, \sshkvu,u\big)\rho^{-1}\big(\hkv,k, \sskv, v\big)$$
 	
 	$$=\hspace{-4ex}\prod\limits_{\dgreen (u,_{_H} v) \in [G \backslash {(U\times_H V)}]}\hspace{-4ex}\mathfrak{l}\Big(\ghkvu,f\big(\sshkvu,_{_H}\sskv\big),\fp\big(u,\sh v\big)\Big)(\lambda\times\rho)^{-1}{\dgreen\Big(\ghkvu,k,\big(\sshkvu,\sh \sskv\big),(u,\sh v)\Big)}$$

 	$$= \hspace{-4ex} \prod\limits_{\dgreen(u,\sh v) \in [G \backslash {(U\times_H V)}]}\hspace{-4ex} \mathfrak{l}\Big(g_{k,(u,\sh v)}w,f\big(\sigma_k(u,\sh v)\big),\fp\big(u,\sh v\big)\Big)(\lambda\times\rho)^{-1}{\dgreen \Big(g_{k,(u,\sh v)}w,k,\big(\sshkvu,\sh \sskv\big),(u,\sh v)\Big)}$$

 	$$=  \hspace{-4ex}\prod\limits_{\dgreen(u,\sh v) \in [G \backslash {(U\times_H V)}]}\hspace{-4ex}\mathfrak{l}\Big(\ g_{k,(u,\sh v)},f\big(\sigma_k(u,\sh v)\big),\fp\big(u,\sh v\big)\Big)
 	(\lambda\times\rho)^{-1}{\dgreen\Big(g_{k,(u,\sh v)},k,\big(\sshkvu,\sh \sskv\big),(u,\sh v)\Big)}
 	$$
 	
 	$$= \mathfrak{L}_{\dgreen U\times_H V}(X,\mathfrak{l}).$$
\end{proof}
{\dgreen
\begin{rem} \label{non free} The following example shows that the assumption that $V$ is left free seems to be necessary for Assertion 5. Suppose that $H=N\rtimes K$ is a semidirect product of a normal subgroup $N$ with $K$. Let $G$ be the group $K$, viewed as a subgroup of $H$. Let moreover $U$ be the set $H$, viewed as a $(G,H)$-biset by left and right multiplication, and let $V$ be the set $K$, acted on by $K$ on the right by multiplication, and by $H$ on the left by projection to $K=H/N$, followed by multiplication in $K$. Let moreover $\lambda$ and $\rho$ be equal to the trivial functor on $\widehat{U}$ and $\widehat{V}$, respectively.\par
Then $U_\lambda\circ_HV_\rho=U\times_HV$, as $\lambda$ and $\rho$ are both trivial. Moreover $U\times_HV=G\times_HK$ is equal to the identity $(K,K)$-biset (this makes sense since $G=K$), so $T_{U\times_HV,\lambda\times \rho}=T_{\id_K,1}$ is the identity functor, by Assertion~4.\par
On the other hand $G\dom U=K\dom (NK)\cong N$, and $H\dom V$ has cardinality 1. So in the computation of the functor $\mathcal{L}_{U,1}$ appearing in $T_{U,1}(X,\mathfrak{l})$, we have a product of values of~$\mathfrak{l}$, indexed by $N$. So the composition $T_{V,1}\circ T_{U,1}$ cannot act in general as the identity on $(X,\mathfrak{l})$, if $N$ is non trivial. Hence $T_{V,\lambda}\circ T_{U,\rho}\neq T_{U\times_HV,\lambda\times \rho}$ in this situation.
\end{rem}
}
\begin{rem} \label{generalize}Let $G$ and $H$ be finite groups, and $U$ be a (finite) $(G,H)$-biset. Then one can check that the diagram
$$\xymatrix{
G\hbox{\sf-poset}\ar[r]^-{\tau_G}\ar[d]^-{T_U}&_CMG\hbox{\sf-poset}\ar[d]^-{T_{U,1_U}}\\
H\hbox{\sf-poset}\ar[r]^-{\tau_H}&_CMH\hbox{\sf-poset}\\
}
$$
of categories and functors is commutative, up to isomorphism, where the functor $T_U$ on the left is the usual generalized tensor induction functor for $G$-posets.
\end{rem}
\begin{lemma} Let $G$ and $H$ be finite groups, and let $(U,\lambda)$ be a  $C$-monomial $(G,H)$-biset. 
    Then there exists a unique map 
	$$\mathcal{T}_{U, \lambda} : B_C(G) \rightarrow B_C(H)$$ such that
	$\mathcal{T}_{U, \lambda}(\Lambda_{(X, \mathfrak{l})}) = \Lambda_{T_{U, \lambda}(X, \mathfrak{l})}$
	for any finite $C$-monomial $G$-poset $(X, \mathfrak{l})$.
\end{lemma}

\begin{proof} We show that if $(X,\mathfrak{l})$
	and $(Y,\mathfrak{m})$ are finite $C$-monomial $G$-posets such that if
	$\Lambda_{(X,\mathfrak{l})} = \Lambda_{(Y,\mathfrak{m})}$ in $B_C(G)$,
	 then $\Lambda_{T_{U, \lambda}(X,\mathfrak{l})} = \Lambda_{T_{U, \lambda}(Y,\mathfrak{m})}$ in $B_C(H).$
	 So it's enough to show that {\rouge $\chi\Big(T_{U,\lambda}\big((X, \mathfrak{l})\big)^{K, \theta}\Big)= \chi\Big(T_{U,\lambda}\big((Y, \mathfrak{m})\big)^{K, \theta}\Big)$}, by Lemma \ref{Lefeu} for any
	  $(K, \theta)$ of $\ch(G)$.
	 
	Let $u \in [G\backslash U /K]$, {\rouge $k \in K$} and  $t \in [K\cap G^u \backslash K]$ then there exist
	a unique $\sskut \in [G\backslash U]$ and some
$\gkut \in G$   such that 
$$utk= \gkut\sskut.$$ 
Also there exist some $c_{k,t} \in K\cap G^u$ and a unique $\tau_k(t)\in [K\cap G^u \backslash K]$
such that 
$$tk= c_{k,t}\tau_k(t).$$  
Since $c_{k,t} \in K\cap G^u$, there exists
$\gamma_{k,t,u}\in G$ such that 
$$uc_{k,t}= \gamma_{k,t,u}u.$$ 
Now
$$utk= uc_{k,t}\tau_k(t)= \gamma_{k,t,u}u\tau_k(t)=\gkut\sskut.$$
So $\sskut= u{\dgreen \tau_k(t)}$ and there exists $w \in G_{\sskut}$ such that 
$\gkut= \gamma_{k,t,u} w$. We get the following commutative diagram:
{\dgreen
$$\xymatrix@R=6ex@C=3ex{
\sskut\ar[rr]^-{w}\ar[dr]_-{\gkut}&&\sskut\ar[dl]^-{\gamma_{k,t,u}}\\
&\rouge utk&
}
$$
}
%
%
%
%
%

Now let {\rouge $f \in t_{U,\lambda}(X,\mathfrak{l})^{K, \theta}$} and $k \in K$ such that
$kf=f$. Note that since $f$ is $K$-fixed, we have 
{\dgreen 
$$f(utk)=f(ut)=f(u)=\gamma_{k,t,u}f\big(u\tau_k(t)\big)=\gamma_{k,t,u}f(u),$$
so $\gamma_{k,t,u} \in G_{f(u)}$. Hence 
$$f\big(\sskut\big)= f\big(u\tau_k(t)\big)= \gamma_{k,t,u}f(utk)=\gamma_{k,t,u}f(u)=f(u).$$
}
Let $\gamma_{k,u}= \prod\limits_{\rouge t \in [K\cap G^u\backslash K]}\gamma_{k,t,u}$ and 
$\phi_u(k)= \prod\limits_{\rouge t \in [K\cap G^u\backslash K]}\lambda^{-1}(\gamma_{k,t,u},k,\sskut, ut)$. 
Then

\begin{align*}
\mathfrak{L}(k,f,f)&= \prod\limits_{u \in [G \backslash U]}\mathfrak{l}\Big(g_{k,u}, f\big(\sigma_k(u)\big), f\big(u\big)\Big)\lambda^{-1}\big(g_{k,u},k, \sigma_k(u),u\big)\\
&=\prod\limits_{\substack{ u \in [G\backslash U /K]\\ {\rouge t \in [K\cap G^u\backslash K]}}}
\mathfrak{l}\Big(\gkut, f\big(\sskut\big), f\big(ut\big)\Big)\lambda^{-1}\big(\gkut,k,\sskut, ut\big)\\
&=\prod\limits_{\substack{ u \in [G\backslash U /K]\\ {\rouge t \in [K\cap G^u\backslash K]}}}
\mathfrak{l}\Big(\gamma_{k,t,u}w, f\big(\sskut\big), f\big(ut\big)\Big)\lambda^{-1}\big(\gamma_{k,t,u}w,k,\sskut, ut\big) \\
&=\prod\limits_{\substack{ u \in [G\backslash U /K]\\ {\rouge t \in [K\cap G^u\backslash K]}}}
\mathfrak{l}\Big(\gamma_{k,t,u}, f\big(\sskut\big), f\big(ut\big)\Big)\lambda^{-1}\big(\gamma_{k,t,u},k,\sskut, ut\big) \\
&= \prod\limits_{u \in [G\backslash U /K]} \mathfrak{l}_{f(u)}(\gamma_{k,u})\phi_u(k)\\
&=\theta(k).\\
\end{align*}

Let $\Xi$ be the family of the sets $\xi= \{\xi_u\}_{u \in [G\backslash U/ K]}$ where 
$\xi_u: {^uK} \rightarrow C$ is a character such that $\ssres^{G_{f(u)}}_{^uK}(\mathfrak{l}_{f(u)})= \xi_u$ and
$$\theta(k)=\!\!\!\!\prod\limits_{u \in [G\backslash U /K]}\!\!\!\! \xi_u(\gamma_{k,u})\phi_u(k)$$
for all $k\in K$ and $u \in [G\backslash U/ K]$.

We claim that

$$T_{U, \lambda}(X,\mathfrak{l})^{K, \theta}= \bigsqcup\limits_{\xi \in \Xi}\prod\limits_{u \in [G\backslash U/ K]}(X, \mathfrak{l})^{^uK, \xi_u}.$$

 Let $f \in T_{U, \lambda}(X, \mathfrak{l})^{K, \theta}$,
then $f(guk)=gf(u)$ for all $g \in G,\,u \in U,$
and $k \in K$. So to determine $f$, it's enough to know 
$f(u)$ for $u \in [G\backslash U/ K]$. Let $f(u)=x_u$.
Then $\ssres^{G_{x_u}}_{^uK}\mathfrak{l}_{x_u} \in \{\xi_u\}_{u \in [G\backslash U/ K]}$
for some $\{\xi_u\}_{u \in [G\backslash U/ K]} \in \xi$.

Conversely, let us choose $x_u \in X$ for any $u \in [G\backslash U /K]$. Let
$v \in V$ then $v=guk$ for some $g \in G$, for some $k \in K$
and a unique $u \in [G\backslash U /K]$. We set $f(v)=gx_u.$ 
Now $f$ is well defined if and only if $gx_u=x_u$ whenever $g \in {^uK}$ or equivalently
$x_u \in X^{^uK}$. We want that $f \in T_{U, \lambda}(X,\mathfrak{l})^{K,\theta}$. If
$x_u \in (X, \mathfrak{l})^{^uK, \xi_u}$ then $\ssres^{G_{x_u}}_{^uK}(\mathfrak{l}_{x_u})= \xi_u$
and 
$$\theta(k)=\!\!\!\!\prod\limits_{u \in [G\backslash U /K]}\!\!\!\! \xi_u(\gamma_{k,u})\phi_u(k).$$ 
So $f \in T_{U, \lambda}(X, \mathfrak{l})^{K, \theta}$.

Now using [\cite{BIS}, Lemma 11.2.9] we get

$$\chi\Big(T_{U, \lambda}(X,\mathfrak{l})^{K, \theta}\Big)= \sum\limits_{\xi \in \Xi}\prod\limits_{u \in [G\backslash U/ K]}\chi\big((X, \mathfrak{l})^{^uK, \xi_u}\big).$$

Thus, if $\Lambda_{(X,\mathfrak{l})} = \Lambda_{(Y,\mathfrak{m})}$ then $\Lambda_{T_{U, \lambda}(X,\mathfrak{l})} = \Lambda_{T_{U, \lambda}(Y,\mathfrak{m})}$. 
So we can define a map 
$$\mathcal{T}_{U, \lambda} : B_C(G) \rightarrow B_C(H)$$ such that
$\mathcal{T}_{U, \lambda}(a)= \Lambda_{T_{U, \lambda}(X,\mathfrak{l})}$ where $(X,\mathfrak{l})$ 
is a $C$-monomial $G$-poset such that
$a= \Lambda_{(X,\mathfrak{l})}$, as in Corollary \ref{lefex}.
\end{proof}

\begin{prop}\label{units}
Let $G$ and $H$ be finite groups, and let $(U,\lambda)$ be a $C$-monomial
$(G,H)$-biset. 
\begin{enumerate}
	\item $\mathcal{T}_{U, \lambda}([G,1_G]_G) = [H,1_H]_H.$
	\item $\mathcal{T}_{U, \lambda}(ab) = \mathcal{T}_{U, \lambda}(a)\mathcal{T}_{U, \lambda}(b)$, for any $a,\, b \in B_C(G).$
	\end{enumerate}
	In particular, the restriction of $\mathcal{T}_{U, \lambda}$ to $B_C(G)^\times$ is a group
homomorphism 
$$\mathcal{T}_{U, \lambda}^\times :B_C(G)^\times \rightarrow B_C(H)^\times.$$
\end{prop}
\begin{proof}	
\begin{enumerate}
\item Consider the $C$-monomial $G$-poset $(\bullet, 1)$ then clearly $\Lambda_{(\bullet, 1)}=
[H,1_H]_H$. So using the first assertion of Proposition \ref{PM} we get
$$\mathcal{T}_{U, \lambda}([G,1]_G)= \Lambda_{T_{U, \lambda}(\bullet, 1)}=
 \Lambda_{(\bullet, 1)}= [H,1_H]_H.$$
\item Let $a,\,b \in B_C(G)$ then by Corollary \ref{lefex} there exist $C$-monomial $G$-posets $(X, \mathfrak{l})$
and $(Y,\mathfrak{m})$ such that $\Lambda_{(X,\mathfrak{l})}=a$ and $\Lambda_{(Y, \mathfrak{m})}=b$. Then 
$$\mathcal{T}_{U, \lambda}(ab)= \mathcal{T}_{U, \lambda}(\Lambda_{(X,\mathfrak{l})}\Lambda_{(Y, \mathfrak{m})})= \mathcal{T}_{U, \lambda}(\Lambda_{X\times Y,\mathfrak{l}\times\mathfrak{m}})= \Lambda_{T_{U, \lambda}(X\times Y,\mathfrak{l}\times\mathfrak{m})}$$
$$=\Lambda_{T_{U, \lambda}(X,\mathfrak{l})}\Lambda_{T_{U, \lambda}(Y, \mathfrak{m})}=\mathcal{T}_{U, \lambda}(a)\mathcal{T}_{U, \lambda}(b)$$
\end{enumerate}

\end{proof}

\begin{prop} \label{depends on U}Let $G$, $H$, and $K$ be finite groups.
\begin{enumerate}
	\item Let $\id_G$ stand for the identity $(G, G)$-biset. Then $\mathcal{T}_{\id_G, 1_G}$
	is the identity map of $B_C(G)$. 
	\item Let $(U,\lambda)$ and $(U^\prime, \lambda^\prime)$ be $C$-monomial $(G, H)$-bisets. Then for any $a \in B_C(G)$
\[\mathcal{T}_{U\sqcup U^\prime, \lambda\sqcup\lambda^\prime}(a)=\mathcal{T}_{U, \lambda}(a)\mathcal{T}_{U^\prime, \lambda^\prime}(a).\]
	\item Let $(U, \lambda)$ be a $C$-monomial $(G, H)$-biset and let $(V, \rho)$ be a monomial left free $(H, K)$-biset then
   $$\mathcal{T}_{V, \rho} \circ \mathcal{T}_{U, \lambda} = \mathcal{T}_{U \times_H V, \lambda \times \rho}.$$
\end{enumerate}	
	
\end{prop}
\begin{proof}
		Let $a \in B_C(G)$ then by Corollary \ref{lefex} there exists a $C$-monomial $G$-poset $(X,\mathfrak{l})$
	such that $a= \Lambda_{(X,\mathfrak{l})}.$ 
	\begin{enumerate}
	 \item  Using the third assertion of Proposition \ref{PM}, we get
	$$\mathcal{T}_{\id_G,1_G}(a)=\mathcal{T}_{\id_G, 1_G}(\Lambda_{(X,\mathfrak{l})})
	=\Lambda_{T_{{\id_G, 1_G}}(X,\mathfrak{l})}= \Lambda_{(X, \mathfrak{l})}= a$$
	\item Using the second assertion of Proposition \ref{PM}, we get
	\begin{align*}
\mathcal{T}_{U\sqcup U^\prime, \lambda\sqcup\lambda^\prime}(a)&=\mathcal{T}_{U\sqcup U^\prime, \lambda\sqcup\lambda^\prime}(\Lambda_{(X,\mathfrak{l})})
	= \Lambda_{T_{U\sqcup U^\prime, \lambda\sqcup\lambda^\prime}(X,\mathfrak{l})}\\
&= \Lambda_{T_{U,\lambda}(X, \mathfrak{l})\times T_{U^\prime, \lambda^\prime}(X,\mathfrak{l})} = \Lambda_{T_{U,\lambda}(X, \mathfrak{l})}\Lambda_{T_{U^\prime, \lambda^\prime}(X, \mathfrak{l})}\\
	&=\mathcal{T}_{U,\lambda}(\Lambda_{(X,\mathfrak{l})})\mathcal{T}_{U^\prime, \lambda^\prime}(\Lambda_{(X,\mathfrak{l})})
	=\mathcal{T}_{U,\lambda}(a)\mathcal{T}_{U^\prime, \lambda^\prime}(a).
\end{align*}
	\item Using the fourth assertion of Proposition \ref{PM}, we get
	$$\mathcal{T}_{V,\rho}\circ\mathcal{T}_{U, \lambda}(a)=\mathcal{T}_{V,\rho} \circ\mathcal{T}_{U,\lambda}(\Lambda_{(X,\mathfrak{l})})
	= \Lambda_{T_{V,\rho} \circ T_{U,\lambda}(X,\mathfrak{l})}$$
	$$= \Lambda_{T_{U\times_HV, \lambda\times\rho}(X,\mathfrak{l})}=\mathcal{T}_{U\times_HV, \lambda\times\rho}(\Lambda_{(X,  \mathfrak{l})})=\mathcal{T}_{U\times_HV, \lambda\times\rho}(a).$$
	
		\end{enumerate}
	\end{proof}
\begin{cor} Let $G$ and $H$ be finite groups. The map $(U,\lambda)\mapsto \mathcal{T}_{U,\lambda}^\times$ of Proposition~\ref{units} extends to a bilinear map
$$B_C(G,H)\times B_C(G)^\times\to B_C(H)^\times.$$
\end{cor}
\begin{proof} This follows from Assertion 2 of Proposition~\ref{PM} and Assertion 2 of Proposition~\ref{depends on U}, and from the fact that the map $\mathcal{T}_{(U,\lambda)}^\times$ depends only on the isomorphism class of $(U,\lambda)$. \end{proof}
\begin{rem} It follows from Remark~\ref{generalize} that if $U$ is a finite $(G,H)$-biset, the square
$$\xymatrix{
B(G)\ar[r]^-{t_G}\ar[d]^-{\mathcal{T}_U}&B_C(G)\ar[d]^-{\mathcal{T}_{U,1_U}}\\
B(H)\ar[r]^-{t_H}&B_C(H)\\
}
$$
of groups and multiplicative maps, is commutative, where $\mathcal{T}_U$ on the left is the usual generalized tensor induction map for Burnside rings, and the horizontal maps $t_G$ and $t_H$ are the ring homomorphisms induced by the functors $\tau_G$ and $\tau_H$. 
\end{rem}
\section*{Acknowledgement} 
The second author was granted 
the Fellowship Program for Abroad Studies 2214-A 
by the Scientific and Technological Research Council of Turkey (T\"ubitak).
The second author also wishes to thank LAMFA for their
hospitality during the visit.

\begin{flushleft}
Serge Bouc\\
CNRS-LAMFA, Universit\'e de Picardie - Jules Verne\\
33 rue St Leu\\
80039 - Amiens - France\\
email: serge.bouc@u-picardie.fr
\end{flushleft}

\begin{flushleft}
Hatice Mutlu\\
Department of Mathematics, Bilkent University\\
06800 - Bilkent, Ankara - Turkey\\
email: hatice.mutlu@bilkent.edu.tr
\end{flushleft}


\begin{bibdiv}
	\begin{biblist}
		
		\bib{FIB}{article}{ author={Barker, L.}, title={ Fibred permutation sets and the idempotents and units of monomial
				Burnside rings}, journal={ Journal of Algebra}, year={2004}, volume={281}, pages= {535-566}}
			
		\bib{BIS}{book}{author={Bouc, S.}, title={Biset Functors for Finite Groups, Lecture Notes in Math.}, date={2010}, volume={1990}, publisher={Springer, Berlin,}}
		
		
		\bib{BURN}{article}{ author={Bouc, S.}, title={Burnside rings}, journal={Handbook of algebra}, year={2000}, volume={2}, pages={739- 804}}
		
		\bib{IND}{article}{ author={Carman, R.}, title={Unit Groups of Representation Rings and their Ghost Rings as Inflation Functors}, journal={Journal of Algebra}, year={2018}, volume={498}, pages={263- 293}}
		
		\bib{MON}{article}{ author={Dress, A.W.M}, title={The ring of monomial representations, I. Structure theory}, journal={Journal of Algebra}, year={1971}, volume={18}, pages={153-157}}

		\bib{quillen}{book}{ author={Quillen, D.}, title={Higher algebraic K-theory, Lecture Notes in Math.}, year={1973}, volume={341}, pages={85-147}, publisher={Springer, Berlin}}
	
		\bib{LEF}{article}{author={Th{\'e}venaz, J.}, title={Permutation representation arising from simplicial complexes}, journal={ J. Combin. Theory}, year={1987},
	    volume={46 Ser. A.},pages={122–155}}
			\end{biblist}
\end{bibdiv}
\end{document}